\documentclass[11pt,a4paper]{article}
\usepackage{amsmath,amssymb,amsthm,graphicx,color,bbm,array}
\usepackage[left=1in,right=1in,top=1in,bottom=1in]{geometry}
\usepackage{cite}
\usepackage{hyperref} 
\usepackage{enumitem}
\usepackage{tikz,pgfplots,pgfplotstable}
\usetikzlibrary{decorations.markings}
\usepackage{caption,subcaption}
\usepackage[british]{babel}
\usepackage[hang,ragged,flushmargin]{footmisc}

\hypersetup{
    colorlinks=false,
    pdfborder={0 0 0},
}

\pgfplotsset{compat=1.16}

\newtheorem{theorem}{Theorem}[section]
\newtheorem{corollary}[theorem]{Corollary}
\newtheorem{proposition}[theorem]{Proposition}
\newtheorem{lemma}[theorem]{Lemma}

\numberwithin{equation}{section}

\theoremstyle{definition}

\theoremstyle{remark}
\newtheorem{remark}[theorem]{Remark}

\newtheorem*{remark*}{Remark}

 \allowdisplaybreaks

\newcommand{\1}[1]{{\mathbbm{1}\mkern -1.5mu}{\{#1\}}}
\newcommand{\2}[1]{{\mathbbm{1}}_{#1}}

\newcommand{\R}{{\mathbb R}}

\newcommand{\Z}{{\mathbb Z}}
\newcommand{\N}{{\mathbb N}}

\newcommand{\ZP}{{\mathbb Z}_+}
\newcommand{\RP}{{\mathbb R}_+}

\newcommand{\dtv}{d_\text{TV}}

\DeclareMathOperator{\Exp}{\mathbb{E}}
\let\Pr\relax
\DeclareMathOperator{\Pr}{\mathbb{P}}

\let\Re\relax
\DeclareMathOperator{\Re}{{\mathrm{Re}}}

\DeclareMathOperator{\Var}{\mathbb{V}ar}

\DeclareMathOperator*{\argmax}{arg \mkern 1mu \max}

\newcommand{\eps}{\varepsilon}

\newcommand{\re}{{\mathrm{e}}}
\newcommand{\ud}{{\mathrm d}}
\newcommand{\unif}[2]{{\mathrm{Unif}[#1,#2]}}
\newcommand{\Bin}[2]{{\mathrm{Bin}(#1,#2)}}

\newcommand{\cE}{{\mathcal E}}
\newcommand{\cF}{{\mathcal F}}

\newcommand{\cK}{{\mathcal K}}
\newcommand{\cL}{{\mathcal L}}

\newcommand{\cN}{{\mathcal N}}

\newcommand{\cT}{{\mathcal T}}

\newcommand{\as}{\ \text{a.s.}}

\newcommand{\bigmid}{\; \bigl| \;}
\newcommand{\Bigmid}{\; \Bigl| \;}
\newcommand{\biggmid}{\; \biggl| \;}

\newcommand{\barPhi}{{\overline{\Phi}}}

\newcommand{\eqd}{\overset{d}{=}}

\makeatletter
\def\namedlabel#1#2{\begingroup  
    (#2)%
    \def\@currentlabel{#2}%
    \phantomsection\label{#1}\endgroup
}
\makeatother

\newlist{myenumi}{enumerate}{10}
\setlist[myenumi]{leftmargin=0pt, labelindent=\parindent, listparindent=\parindent, labelwidth=0pt, itemindent=!, itemsep=1pt, parsep=4pt}

\newlist{thmenumi}{enumerate}{10}
\setlist[thmenumi]{leftmargin=0pt, labelindent=\parindent, listparindent=\parindent, labelwidth=0pt, itemindent=!}

\title{Rates of convergence for extremal spacings in Kakutani's\\ random interval-splitting process}

\author{Fraser Daly\footnote{\raggedright Department of Actuarial Mathematics and Statistics, and the Maxwell Institute for Mathematical Sciences, Heriot--Watt University, Edinburgh EH14 4AS; \href{mailto:F.Daly@hw.ac.uk}{\texttt{F.Daly@hw.ac.uk}}} \and Andrew Wade\footnote{\raggedright Department of Mathematical Sciences, Durham University, Durham DH1 3LE; \href{mailto:andrew.wade@durham.ac.uk}{\texttt{andrew.wade@durham.ac.uk}}}}

\begin{document}

\maketitle

\begin{abstract}
Kakutani's random interval-splitting process iteratively divides, via a uniformly
random splitting point, the largest sub-interval in a partition of the unit interval. The length of the longest sub-interval after $n$ steps, suitably centred and scaled, is known to satisfy a central limit theorem as $n \to \infty$. We provide a quantitative (Berry--Esseen) upper bound for the finite-$n$ approximation in the central limit theorem, with conjecturally optimal rates in~$n$. We also prove convergence to an exponential distribution for the length of the smallest sub-interval, with quantitative bounds. The Kakutani process can be embedded in certain branching and fragmentation processes, and we translate our results into that context also. Our proof uses conditioning on an intermediate time, a conditional independence structure for statistics involving small sub-intervals, an Hermite--Edgeworth expansion, and moments estimates with quantitative error bounds.
\end{abstract}

%\tableofcontents
%\medskip


\smallskip

 (Secondary).

\section{Kakutani's interval-splitting process}

\subsection{Main results}
\label{sec:main-result}

The subject of this paper is the following random process 
which takes values on partitions of the unit interval
and evolves by successive uniform binary splitting of the maximal interval, attributed to Kakutani. 
Start with the unit interval $[0,1]$; at each subsequent step, choose the largest of the current collection of intervals, and split it into
two random subintervals by inserting a uniform random splitting point, independently of previous steps. 
Ties can be broken arbitrarily,
but, with probability~$1$, they do not occur. 

We give some slightly more formal definitions, and some 
historical context and motivation, in Section~\ref{sec:background} below. Our main result, Theorem~\ref{thm:clt},  concerns the $n \to \infty$ asymptotics of the random variable $M_n$, the length of the largest among the $n+1$ subintervals in the partition
resulting from $n$~steps of the process  described above.

Since the sum of all the subintervals in the partition is always~$1$, it is clear that $M_n > \frac{1}{n+1}$, a.s., for every $n \in \N := \{1,2,3,\ldots\}$. 
The strong law of large numbers
\begin{equation}
    \label{eq:slln}
    \lim_{n \to \infty}   n M_n   = 2 , \as
\end{equation}
is due to Lootgieter
(Corollary 1.2 of~\cite[p.~397]{lootgieter-77b}) and Pyke (Lemma~1 in~\cite[p.~159]{pyke}).
Over 20 years later, Pyke and van~Zwet~\cite[p.~414]{pvz}
obtained the following central limit theorem (CLT).
Let $\Phi$ denote the cumulative distribution function of the standard normal distribution, i.e., 
$\Phi (x) := (2\pi)^{-1/2} \int_{-\infty}^x \re^{-z^2/2} \ud z$ for $x \in \R$.

\begin{proposition}[Pyke \& van Zwet, 2004]
\label{prop:M-clt}
Let $\sigma^2 := 16 \log 2 - 10 \approx 1.09035$.
Then 
\[ 
\lim_{n \to \infty} \Pr \left(  \sqrt{ \frac{n^3}{\sigma^2}} \left( M_n - \frac{2}{n} \right)  \leq x \right) = \Phi (x), \text{ for every } x \in \R.\]
\end{proposition}
\begin{remark}
Via an embedding we describe in Section~\ref{sec:branching}, and an inversion we describe in Section~\ref{sec:threshold-times}, Proposition~\ref{prop:M-clt} also follows from earlier work of Sibuya \& Itoh~\cite{si}: see Remark~\ref{rem:N-CLT}.
\end{remark}

Our main result  gives Berry--Esseen bounds on the rate of convergence
in the CLT of Proposition~\ref{prop:M-clt}.

\begin{theorem}
\label{thm:clt}
There exists a constant $C \in \RP := [0,\infty)$ such that,  
\begin{equation}
\label{eq:main-clt}
\sup_{x \in \R} \left| \Pr \left( \sqrt{ \frac{n^3}{\sigma^2}} \left( M_n - \frac{2}{n} \right)  \leq x  \right) - \Phi (x) \right| \leq \frac{C}{\sqrt{n}}  , \text{ for every } n \in \N , \end{equation}
where $\sigma^2$ is as defined in Proposition~\ref{prop:M-clt}.
\end{theorem}

\begin{remark}
\label{rems:clt}
   We conjecture that the $O(n^{-1/2})$ error bound in Theorem~\ref{thm:clt} is optimal,
   as is generic in the 
   classical Berry--Esseen theorem (see~e.g.~\cite[\S7.6]{gut}). While we do not have a proof  of a lower bound of matching order to the upper bound in~\eqref{eq:main-clt},  Figure~\ref{fig1} presents some simulation evidence that appears to support this conjecture. 
\end{remark}

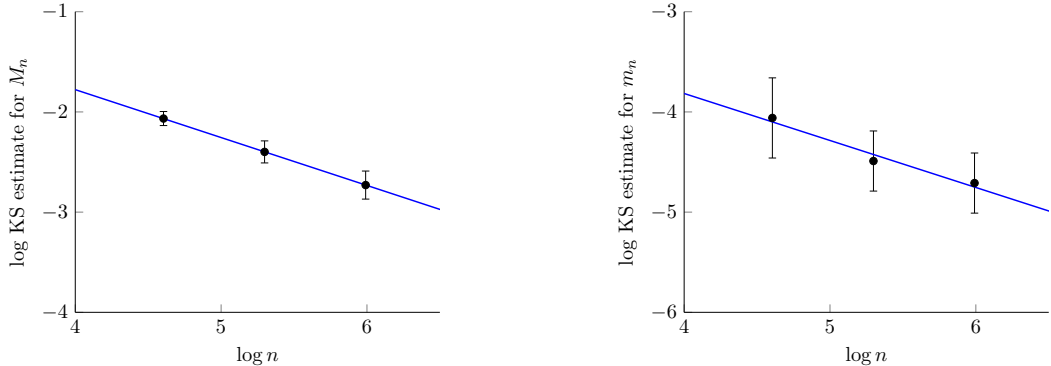
\begin{figure}[t]
    \centering
    \begin{tikzpicture}[scale=0.7]
    \begin{axis}[
    xtick pos=left,
    ytick pos=left,
    axis y line*=left,
    axis x line*=left,
    ylabel near ticks,
    xlabel near ticks,
    xlabel=$\log n$,
    ylabel=$\log$ KS estimate for $M_n$,
    xmin=4,xmax=6.5,
    ymin=-4, ymax=-1,
    xtick={4, 5, 6},
    xticklabels={4, 5, 6,  highres},
    ]
    \addplot[
    scatter/classes={d={black}},
    scatter,
    only marks,
    visualization depends on=\thisrow{ey} \as \myshift,
    every node near coord/.append style = {shift={(axis direction
    cs:0,\myshift)}},
    scatter src=explicit symbolic,
    nodes near coords*={\Label},
    visualization depends on={value \thisrow{label} \as \Label},
    ]%
    plot [error bars/.cd, y dir = both, y explicit]
    table[meta=class, x=x, y=y, y error=ey]{
        x       y     ey    class label
        4.605  -2.066 0.07  d   {}  
        5.298  -2.399 0.11  d   {}  
        5.991  -2.730 0.14  d   {}  
   %     6.685  -3.060 0.20  d    {}
    };
    \addplot[thick, draw=blue, mark=none,domain={4:7.5}] {-.478*x+0.1343};
    \end{axis}
    \end{tikzpicture}
    \hspace{2cm}
      \begin{tikzpicture}[scale=0.7]
    \begin{axis}[
    xtick pos=left,
    ytick pos=left,
    axis y line*=left,
    axis x line*=left,
    ylabel near ticks,
    xlabel near ticks,
    xlabel=$\log n$,
    ylabel=$\log$ KS estimate for $m_n$,
    xmin=4,xmax=6.5,
    ymin=-6, ymax=-3,
    xtick={4, 5, 6},
    xticklabels={4, 5, 6, highres},
    ]
    \addplot[
    scatter/classes={d={black}},
    scatter,
    only marks,
    visualization depends on=\thisrow{ey} \as \myshift,
    every node near coord/.append style = {shift={(axis direction
    cs:0,\myshift)}},
    scatter src=explicit symbolic,
    nodes near coords*={\Label},
    visualization depends on={value \thisrow{label} \as \Label},
    ]%
    plot [error bars/.cd, y dir = both, y explicit]
    table[meta=class, x=x, y=y, y error=ey]{
        x   y   ey  class   label
        4.605  -4.06 0.4  d   {}     
        5.298  -4.49 0.3 d {}
        5.991  -4.71 0.3  d   {} 
 %       6.685  -4.80 0.31  d    {}  
    };  
    \addplot[thick, draw=blue, mark=none,domain={4:7.5}] {-.469*x-1.94};
    \end{axis}
    \end{tikzpicture}
        \caption{Simulated Kolmogorov--Smirnov (KS) distances for extremal spacings. For each $n = 2^k \cdot 100$ ($k \in \{0,1,2\}$), $10^4$ Monte Carlo samples of $M_n, m_n$ were used to estimate the empirical distribution of $\sqrt{ n^3/\sigma^2} (M_n - (2/n))$ (\emph{left}) and $n^2 m_n /2$ (\emph{right}) and hence, after  $10^3$ replications, obtain estimates of the $\log$ of the KS distance to the standard Gaussian (\emph{left}) and unit-mean exponential (\emph{right}), respectively. That is, pictured are estimates for the $\log$ of the left-hand side of~\eqref{eq:main-clt} (\emph{left}) and~\eqref{eq:small-gap} (\emph{right}), along with indication of Monte Carlo spread (error bars). The $\log$-$\log$ plots show a least-squares fit line with slope $\approx -0.48$ (\emph{left}) and $-0.47$ (\emph{right}), consistent with $O(n^{-1/2})$ being the best possible polynomial rate in Theorems~\ref{thm:clt} and~\ref{thm:small-gap}.} \label{fig1}
\end{figure}

A second limit theorem that we deduce from some of the structural  results that we 
develop in this paper concerns the length $m_n$ of the \emph{smallest gap} after $n$ steps of Kakutani's interval-splitting process. In particular, the next result shows that $n^2 m_n/2$ converges in distribution to a unit-mean exponential distribution as $n \to \infty$.

\begin{theorem}
\label{thm:small-gap}
There exists a constant $C \in \RP$ such that,  
\begin{equation}
    \label{eq:small-gap}
 \sup_{x \in \RP} \left| \Pr \left(  \frac{ n^2 m_n}{2}   > x  \right) - \re^{-x} \right| \leq \frac{C (1+ \log n )}{\sqrt{n}}  , \text{ for every } n \in \N . \end{equation}
\end{theorem}

\begin{remark}
\label{rems:small-gap}
We suspect that the $n^{-1/2}$ polynomial rate in~\eqref{eq:small-gap} cannot be improved (see Figure~\ref{fig1}),
in contrast to the $n^{-1}$ rate in the analogous limit for the Dirichlet process at~\eqref{eq:dirichlet-min} below, but we are unsure if the $\log n$ factor is sharp. 
Establishing the optimal rate of convergence in Theorem~\ref{thm:small-gap}, including removing any possibly superfluous logarithmic factors and explaining apparent non-linearities in the right-hand plot of Figure~\ref{fig1}, is left as a topic for future work.
\end{remark}

We give an overview of our proof strategy, and the organization of the paper, in Section~\ref{sec:proof-overview} below. In the next section we introduce some notation for formally describing the model.

\subsection{Notation and background}
\label{sec:background}

To discuss the earlier work on the Kakutani model, and for later use in the present paper,
we need some more notation. At time~$n \in\ZP:= \{0\} \cup \N$ (that is, after $n$ splitting events)
the process is represented by the ordered collection
of interval end-points
\begin{equation}
\label{eq:X-notation}
X_{n,0} := 0  < X_{n,1} < \cdots < X_{n,n} < 1 =: X_{n,n+1}.\end{equation}
The associated gap lengths at time $n \in \ZP$ are
\begin{equation}
\label{eq:L-def}
L_{n,i} := X_{n,i} - X_{n,i-1}, \text{ for } 1 \leq i \leq n+1 .\end{equation}
The maximal and minimal gap lengths at time $n \in \ZP$ are, respectively,
\begin{equation}
\label{eq:M-def}
M_n := \max_{1 \leq i \leq n+1} L_{n,i}  ,
\text{ and } m_n:=\min_{1 \leq i \leq n+1} L_{n,i}  .
\end{equation}
The dynamics are driven by  $U_1, U_2, \ldots$,  a sequence of independent $\unif{0}{1}$ random variables.
Let $X_{1,1} := U_1$ and then, recursively, given  
\begin{equation}
\label{eq:L-distinct}
L_{n,1}, \ldots, L_{n,n+1} \text{ all distinct and non-zero}, \end{equation}
define $\ell_n := \argmax_{1 \leq i \leq n+1} L_{n,i}$ (which is uniquely defined).
Then set
\begin{equation}
\label{eq:splitting-recursion}
X_{n+1,i} = \begin{cases} X_{n,i} & \text{for } 0 \leq i \leq \ell_n -1, \\
X_{n, \ell_n-1} + U_{n+1} M_n & \text{for } i = \ell_n , \\
X_{n, i-1}  & \text{for } \ell_n +1 \leq i \leq n+2. \end{cases} \end{equation}
The fact that the successive divisions  are generated by continuous distributions ensures that, a.s., properties~\eqref{eq:X-notation} and~\eqref{eq:L-distinct}
persist  for all $n \in \N$.

Kakutani conjectured   in a 1973 lecture, in response to a question from H.~Araki (see~\cite[p.~341]{slud}, \cite[p.~571]{kanter} and~\cite[p.~258]{af}) that the empirical distribution $\cE_n (x ) := \frac{1}{n} \sum_{i=1}^n \1 { X_{n,i} \leq x}$  of endpoints
is asymptotically uniform, or \emph{equidistributed}, i.e.
\begin{equation}
\label{eq:equidistribution}
 \lim_{n \to \infty} \sup_{x \in [0,1]} \left| \cE_n (x ) - x \right| = 0 , \as \end{equation}
Kakutani~\cite{kakutani} obtained an analogous result for a deterministic 
model in which the $\unif{0}{1}$
splitting distribution is substituted by a fixed parameter $\alpha \in (0,1)$ for the relative location of the division point.
For the $\unif{0}{1}$-splitting process, 
establishing~\eqref{eq:equidistribution}
was posed as a challenge by R.\ Dudley in 1976~\cite[p.~2443]{bfgg}. The result~\eqref{eq:equidistribution} was proved by van Zwet in~\cite{vanzwet} (submitted in February~1977)
and, independently, using similar ideas, by Lootgieter~\cite{lootgieter-77a,lootgieter-77b} (submitted later the same year).
Van~Zwet~\cite[p.~137]{vanzwet} also acknowledges that Koml\'os and Tusn\'ady 
were aware that a proof could be constructed by the same method. 
A paper of Slud~\cite{slud}, submitted in October~1976, asserted a proof of~\eqref{eq:equidistribution}, via a rather different approach,
but was found to contain an error; Slud later produced a correction. Results on empirical distributions extending~\eqref{eq:equidistribution} to a much wider class of interval-division schemes can be found in~\cite{mp1}; see also references therein.

Define, for $n \in \N$ and $y \in \R$,
\begin{equation}
\label{eq:G-n-def}
G_n (y) := \frac{1}{n+1} \sum_{i=1}^{n+1} \1 { (n+1)L_{n,i} \leq y },  
\end{equation}
the distribution function of  
the 
 empirical measure for the normalized gap lengths. Pyke's uniform limit theorem~\cite[Thm.~1, p.~161]{pyke} shows that
\begin{equation}
\label{eq:e-d-uniform}
 \lim_{n \to \infty} \sup_{y \in [0,2]} \left|  G_n (y) - \frac{y}{2} \right| = 0, \as ,
\end{equation}
which shows that a typical gap is approximately $\unif{0}{2/n}$ in distribution. While~\eqref{eq:e-d-uniform} says that there are at most~$o(n)$ gaps of size bigger than $(2+\eps)/n$, $\eps>0$, Proposition~\ref{prop:M-clt} says that even the maximum gap will typically be only of order $n^{-3/2}$ different from~$2/n$.

Before continuing with preparation for the rest of the paper, 
next, in Section~\ref{sec:branching} we describe how our result can be interpreted in the context of a  branching process, and its relationship to some adjacent models.
 
\subsection{Related branching, fragmentation, and parking  processes}
\label{sec:branching}

The recursive interval division of the Kakutani process is reminiscent of branching and fragmentation structures. 
Indeed, the Kakutani process admits several embeddings into (perhaps more familiar) classes of probability models, and is adjacent to several others. In this section we indicate some of these links, where, typically, optimal Berry--Esseen results do not appear to be known, in part to draw attention to scope for possible extensions of the present work.

We believe that Kingman~\cite{kingman} was  the first to  explicitly  use  an embedding into a  branching process to put the strong law~\eqref{eq:slln} for the Kakutani process into a more general framework;
essentially the same construction appears in Sibuya \& Itoh~\cite{si} but without the explicit link to the Kakutani model. Subsequent work has also emphasized correspondences to  fragmentation processes and branching random walks, where considerable technology has been developed. The Kakutani process translates to quite special versions of these general structures, and, as far as we are aware, our main results are not subsumed within the general literature.

 Kingman~\cite{kingman} observed that 
applying the bijective map $x \mapsto \log (1/x)$ to the collection of subinterval lengths in the Kakutani process gives a process on $(0,\infty)$ in which interval splitting translates to additive displacement; there are several ways to exploit this.
%, depending on whether this displacement is interpreted as spatial or temporal. 
The first is Kingman's original embedding.

\paragraph{Total population in a binary branching process.}
In the Kakutani process, since $M_n \to 0$, a.s., every subinterval present in the process at some time $m$, say, will be split at some time $n \geq m$, and hence have length equal to $M_n$. In other words,
$\cL := \{ M_0, M_1, M_2, \ldots \}$ is the collection of all (a.s.~distinct) subinterval lengths observed during the Kakutani process, every $L_{n,i}$ appears in $\cL$, and, typically,  $L_{n,i} = L_{m,j}$ for many pairs $(n,i)$, $(m,j)$. The length $M_n$ is in position $n+1$ in $\cL$ viewed as an ordered list. That is, if $N_u  := \sum_{\ell \in \cL} \1 { \ell > u }$, we have $\{ M_n > u \}$ if and only if $\{ N_u > n \}$.

Applying the map $x \mapsto \log (1/x)$  gives the branching process interpretation as a population
of individuals with birth times indexed by $\cT := - \log \cL \subset \RP$: the ancestor $0  =  - \log M_0 \in \cT$ is born at time $0$, and every individual $t \in \cT$ gives birth to two offspring at two (different, correlated) times distributed as $t - \log U$ and $t - \log (1-U)$, $U \sim \unif{0}{1}$. 
(Throughout the paper, the notation $\unif{a}{b}$ stands for the continuous uniform distribution on interval $[a,b]$, for real $a<b$.)
The quantity $N_{\re^{-t}} = T_t := \# ( \cT \cap [0,t) )$ is the total population number observed before time $t \in \RP$.

This construction, relating  $M_n$ in the Kakutani process to total population size in a Crump--Mode--Jagers branching process with binary offspring and correlated  birth times,
is due to Kingman~\cite{kingman}.  
The following is  a translation of Theorem~\ref{thm:clt} (in fact, the translation is  direct from Theorem~\ref{thm:n-b-e} via the inversion described in Section~\ref{sec:threshold-times} below). 

\begin{corollary}
\label{cor:cmj}
There exists $C \in \RP$ such that, with $\sigma^2$ as defined in Proposition~\ref{prop:M-clt},
\[ \sup_{x \in \R} \left| \Pr \left( \sqrt{ \frac{2}{\sigma^2}} \re^{-t/2} \left(   T_t - 2 \re^t \right)  \leq x  \right) - \Phi (x) \right| \leq  C \re^{-t/2}   , \text{ for every } t \in \RP . \]
\end{corollary}

A different perspective on Kingman's construction, going back to
Sibuya \& Itoh~\cite{si}, interprets $N_u$ as  the height of a \emph{random fragmentation tree}, a subject studied more generally in~\cite{jn} (see also the references therein). That description is very close to the
one we present in Section~\ref{sec:threshold-times}, and so we defer further discussion to Remark~\ref{rem:N-CLT} below.

Recently, for a  general class of Crump--Mode--Jagers processes,
deep results~\cite{henry,iksanov} have been obtained generalizing the convergence in distribution part of Corollary~\ref{cor:cmj}.  Theorem~3.2 of~\cite{henry} is specific to binary branching but does not apply directly to our case, as that work assumes independent birth times for siblings,
while~\cite{iksanov} does admit the correlations present here. Of course, the non-quantitative part of Corollary~\ref{cor:cmj} is already known via  a translation of Proposition~\ref{prop:M-clt},
but the  results of~\cite{henry,iksanov} indicate that there is a  much more general setting  than the Kakutani process to which one might seek to extend the Berry--Esseen results of Theorem~\ref{thm:clt}. We are not aware of any existing Berry--Esseen results in the Crump--Mode--Jagers context, and our approach in the present paper seems to make essential use of features of the Kakutani model. 

\paragraph{Extremum-driven branching random walk.} Applying the bijective map $x \mapsto \log (1/x)$ to the collection of subinterval lengths, but retaining the original time indexing, translates the Kakutani process to the following special type of discrete-time branching random walk on $\RP$. Start with a single particle at the origin. At each step, the leftmost particle is removed, and replaced with two independent offspring, displaced by dependent unit-mean exponential random variables relative to the parent. Let $\ell_n, r_n$ be the location of the leftmost, respectively, rightmost particles after~$n$ branching events. Then $\ell_n = \log (1/M_n)$ and $r_n = \log (1/m_n)$ in terms of $M_n, m_n$ in the Kakutani process described in Section~\ref{sec:main-result}. The following is then a consequence of Theorems~\ref{thm:clt} and~\ref{thm:small-gap}.

\begin{corollary}
\label{cor:brw}
There exists $C \in \RP$ such that, with $\sigma^2$ as defined in Proposition~\ref{prop:M-clt},
\begin{equation}
\label{eq:brw-1}
\sup_{x \in \R} \left| \Pr \left( \sqrt{ \frac{4n}{\sigma^2}} \left( \ell_n - \log (n/2) \right)  \leq x  \right) - \Phi (x) \right| \leq  \frac{C}{\sqrt{n}}  , \text{ for every } n \in \N . \end{equation}
Moreover,
\begin{equation}
\label{eq:brw-2}
\sup_{x \in \R} \left| \Pr \left( r_n - \log ( n^2 /2)  \leq  x  \right) - \exp \bigl( -{\re^{-x}} \bigr) \right| \leq \frac{C (1+ \log n )}{\sqrt{n}}  , \text{ for every } n \in \N . \end{equation}
\end{corollary}

Other  models with branching and rank-dependent dynamics,
but a \emph{fixed} population size, have been studied, motivated, for example, by selective pressures on individuals in an evolving population~\cite{brunet,bg} or on species in an evolving ecosystem~\cite{bak}.

\paragraph{A zero-length slack parking model.}
Fix parameters $x \in (0,\infty)$ and $\ell \in [0,1]$. 
Take a sequence of independent $\unif{0}{x}$ random variables, representing the left endpoints of  length-$\ell$ cars that successively arrive at the kerb $[0,x]$. Each car is allowed to park if and only if (i) its extent is contained in the subset of~$[0,x]$ not already occupied by parked cars, and (ii) the gap in which it parks has length exceeding~1. The process becomes \emph{jammed} once no gap between neighbouring cars has length exceeding~$1$, at which point no more cars can park. Let $P_{\ell,x}$ denote the (random) number of cars parked at jamming. 

The $\ell =1$ version of this model was first studied by R\'enyi~\cite{renyi}, who evaluated the asymptotic parking fraction~$\lim_{x \to \infty} x^{-1} \Exp P_{1,x}$. The generalization $0 \leq \ell < 1$, which introduces some slack around each car, is included in~\cite{mullooly,si}, with the asymptotic parking fraction being obtained in~\cite{ki}. The parking model is part of a wide class of models motivated by various irreversible physical and chemical processes, known more broadly as \emph{random sequential adsorption}. 

When~$\ell =0$ (zero-length cars)  the random variable $P_{0,x}$ has the interpretation as the number of splitting events in the following procedure: starting with the interval $[0,x]$, uniformly split any interval of length exceeding~$1$ until there is no such interval left. It is then not hard to see that $P_{0,x}$ has the same distribution as the number of steps in the Kakutani model until all intervals have length at most $1/x$. This quantity is described in more detail in  Section~\ref{sec:threshold-times} below; specifically, $P_{0,x}$ is $N_{1/x}$ defined at~\eqref{eq:nt-def}. In Theorem~\ref{thm:n-b-e} below we give a Berry--Esseen theorem for $N_t$ as $t \to 0$, and hence for $P_{0,x}$ as $x \to \infty$. 

For $\ell =1$ (the original R\'enyi model)
a CLT for $P_{1,x}$ is due to Dvoretzky \& Robbins~\cite{dr}, and 
a Berry--Esseen theorem for $P_{1,x}$ was obtained by Schreiber, Penrose \& Yukich~\cite{spy}, as a special case of a much more general result, but with $\log$ factors in the rate that  Theorem~\ref{thm:n-b-e} below suggests one might be able to remove, in the one-dimensional case, by adapting the approach of the present paper. More recently, a refined   general approach for Berry--Esseen bounds for functions of Poisson point processes, with presumably-optimal rates, has been given by~\cite{lrp}, but, as far as we know,   has yet to be successfully applied to random sequential adsorption.

\subsection{Inversion and threshold times}
\label{sec:threshold-times}

Associated to the Kakutani process are  the \emph{threshold times} $N_t$, $t \in (0,\infty)$, defined by
\begin{equation}
\label{eq:nt-def}
N_t := \inf\{n\in \ZP :M_n\leq t\}\,;
\end{equation}
since $M_0=1$, we have $N_t =0$ for all $t \geq 1$. Moreover, 
since $\lim_{n \to \infty} M_n =0$ a.s.~(which follows of course from~\eqref{eq:slln}, but also via a short elementary argument, as given by Kingman~\cite[p.~148]{kingman}),
we have $\Pr ( N_t < \infty \text{ for every } t > 0 ) = 1$.
The usefulness of $N_t$ as defined at~\eqref{eq:nt-def} for analyzing $M_n$ is due, firstly, to 
the inversion relation
\begin{equation}
\label{eq:M-N-inverse}
\Pr ( M_n \leq t ) = \Pr (N_t \leq n ), \text{ for every } n \in \N \text{ and } t \in (0,\infty), \end{equation}
and, secondly, a more readily accessible recursive structure. Indeed,
by conditioning on the first split (through the variable $U_1$) we obtain the 
fundamental self-similarity relation
\begin{equation}
\label{eq:rde-U1}
N_t = N^{(1)}_{t/U_1} + N^{(2)}_{t/(1-U_1)} + 1 , \as, \text{ for } 0 < t < 1. \end{equation}
In~\eqref{eq:rde-U1}, the processes $N^{(1)} := ( N^{(1)}_s )_{s \in (0,1)}$ and $N^{(2)} := ( N^{(2)}_s )_{s \in (0,1)}$ are independent of $U_1$ and of each other,
and each has the same distribution as $( N_t )_{t \in (0,1)}$. To see that~\eqref{eq:rde-U1} is true,
observe that to reach time $N_t$ the intervals $[0,U_1]$ and $[U_1, 1]$
undergo \emph{independent} Kakutani processes (and one can choose to execute all splittings
on one side first, for example) but scaled by the relevant length factor, i.e., $U_1$ or $1-U_1$.
When the identification of $U_1$ in~\eqref{eq:rde-U1}
with the generating sequence of the process is not relevant, we can write~\eqref{eq:rde-U1} in distributional form
as
\begin{equation}
\label{eq:rde-U}
N_t \eqd N^{(1)}_{t/U} + N^{(2)}_{t/(1-U)} + 1 , \text{ for } 0 < t < 1, \end{equation}
where, on the right-hand side of~\eqref{eq:rde-U},
$U \sim \unif{0}{1}$,
and $U, N^{(1)}, N^{(2)}$ are independent. The central role of the {threshold times} $N_t$ defined at~\eqref{eq:nt-def} and their associated recursions~\eqref{eq:rde-U1}--\eqref{eq:rde-U} was already identified, independently, by van Zwet~\cite[p.~134]{vanzwet} and Lootgieter~\cite[p.~404]{lootgieter-77a}.

The CLT for $M_n$,  Proposition~\ref{prop:M-clt}, was
obtained by Pyke \&~van Zwet~\cite{pvz} 
via the inversion~\eqref{eq:M-N-inverse}
from a corresponding CLT for $N_t$.

\begin{proposition}[Pyke \& van Zwet, 2004]
\label{prop:N-clt}
Let
 $\sigma^2$ be as in Proposition~\ref{prop:M-clt}.
As $t \to 0$,
$\sqrt{t} \left( N_t - \frac{2}{t} \right)$ converges  to the normal distribution with mean $0$ and variance $\sigma^2/2$.
\end{proposition}

In a similar way, we will obtain our quantitative CLT,
Theorem~\ref{thm:clt}, via an inversion of
a corresponding Berry--Esseen result for $N_t$.

\begin{theorem}
    \label{thm:n-b-e}
There exists a constant $C \in \RP$ such that, for all $t \in (0,\infty)$, 
\begin{equation}
    \label{eq:n-b-e}
 \sup_{x \in \R} \left| \Pr \left( \sqrt{\frac{2 t}{\sigma^2}} \left( N_t -\frac{2}{t} \right) \leq x   \right) - \Phi (x) \right| \leq C \sqrt{ t },
\end{equation}
where $\sigma^2$ is as defined in Proposition~\ref{prop:M-clt}. 
\end{theorem}

\begin{remark} 
\label{rem:N-CLT}
Proposition~\ref{prop:N-clt} is Corollary~3.3 in~\cite[p.~396]{pvz}. By Kingman's embedding (see Section~\ref{sec:branching}), it can also be recovered from the earlier Theorem~2 of~\cite{si}. See also~\cite[pp.~435--6]{jn} for a framing in terms of general CLTs for the height of random fragmentation trees; we do not know of any Berry--Esseen results in that context.
 \end{remark}
 
 \subsection{Overview of the proofs and some further remarks}
 \label{sec:proof-overview}

 \paragraph{Overview of the proofs.}
 The main work of the paper is proving Theorem~\ref{thm:n-b-e}.
 Theorem~\ref{thm:clt} will be deduced from a careful (but not difficult) inversion of Theorem~\ref{thm:n-b-e}, and Theorem \ref{thm:small-gap} from some results established in the course of the proof of Theorem \ref{thm:n-b-e}. The proof of Theorem~\ref{thm:n-b-e} is in part analytical, with some delicate estimates needed to obtain our presumably-optimal rates, and the overall structure is perhaps of interest more broadly, and is broken down into the following main steps.
 \begin{itemize}
     \item The first step in the main line of the  proof is to apply the classical Berry--Esseen theorem to prove (in Section~\ref{sec:berry-esseen}) a conditional Berry--Esseen theorem (Proposition~\ref{lem:b-e-conditional}) for $N_t$ given the first~$n$ steps of the process, where $4 (n+1) t <1$, ensuring that $0 < N_t-n$ can be expressed as a sum of independent variables, using the basic self-similarly~\eqref{eq:rde-U}. Eventually, we will take $n \approx c_1/t$ for a small $c_1$.
     \item The centering and scaling quantities in the conditional Berry--Esseen bound are themselves random variables (functions of $U_1, \ldots, U_n$), related to conditional means and variances, denoted by $R_{n,t}$ and $S_{n,t}$ defined at~\eqref{eq:R-def} and~\eqref{eq:S-def} below. To ``uncondition'' the bound needs detailed information about the joint distribution of $R_{n,t}$ and $S_{n,t}$, summarized in Proposition~\ref{prop:R-S-cross-moments} on their mixed moments. This is proved in Section~\ref{sec:RS-moments}, with groundwork laid in Sections~\ref{sec:small-gaps} and~\ref{sec:berry-esseen} and making use of auxiliary results stated in Appendix \ref{sec:appendix}.
     \item To study the joint distribution of  $R_{n,t}$ and $S_{n,t}$,
     we exploit the fact that both can be expressed as sum-type statistics of \emph{small gaps}, that enjoy a crucial conditional independence structure which we clarify in Section~\ref{sec:small-gaps}. This study of small gaps will also lead to a short proof of Theorem~\ref{thm:small-gap}, given in Section~\ref{sec:smallest-gap}.
     \item The ``unconditioning'' is achieved by a sort of Hermite--Edgeworth expansion, stated in Proposition~\ref{prop:clt-with-error}. Combined with Proposition~\ref{prop:R-S-cross-moments} to control the remaining error terms in the expansion leads to the proof of Theorem~\ref{thm:n-b-e}, given in Section~\ref{sec:hermite}.
     \item To prepare for all of the above, we first collect some results (some known, some new, making use of ideas from~\cite{pvz}) on moments of $N_t$ and $M_n$ in Section~\ref{sec:means}.
 \end{itemize}
Following the proofs of our main theorems, the proofs of Corollaries \ref{cor:cmj} and \ref{cor:brw} are given in Section \ref{sec:hermite}.

\paragraph{Comparison to the Dirichlet process and uniform spacings.}  A natural comparator to the Kakutani problem is the   Dirichlet partition of the interval generated by $n$ $\unif{0}{1}$ random variables, which can also be generated sequentially, like the Kakutani process, but  one splits an interval chosen at random with probability proportional to its length (rather than always the longest). Denote the maximal spacing in the Dirichlet process after $n$~divisions by $M_n^D$. A result of L\'evy from 1939 shows that $n M_n^D - \log n$ has a Gumbel limit, and Slud~\cite{slud} showed that $n M_n^D / \log n \to 1$, a.s., as $n \to \infty$. 

Let $m_n^D$ denote the length of the smallest gap in the $n$-division Dirichlet process. A direct calculation
shows that $\Pr ( m_n^D \geq x ) = (1-(n+1) x)^n$, $x \in [0,\frac{1}{n+1}]$, from which one can show
\begin{equation}
\label{eq:dirichlet-min}
\sup_{x \in \R} \left| \Pr \left( 
 {n^2 m_n^D} \geq x \right) - \re^{-x} 
\right| \leq  \frac{C}{n} , \text{ for every } n \in \N, \end{equation}
and this bound is of the optimal order in~$n$. 
It follows from~\eqref{eq:dirichlet-min} that $n^2 m_n^D$
 converges to a unit-mean exponential, roughly half the smallest gap in the Kakutani process. Some intuition for this comes from observing that in the Dirichlet process, one is typically splitting a gap of length on average half the size of $M_n$, the length split in the Kakutani process.

\paragraph{Other order statistics.} We expect that one can obtain some information about the length of near-maximal gaps, or near-minimal gaps, using our method and some extra work. It would be of interest to obtain results for more general order statistics of spacings, but it is not clear to us how to do this.

\section{Means, variances, and moment bounds}
\label{sec:means}

In this section we study moments of the random variables $M_n$ and $N_t$.
For $t \in (0,\infty)$, let $\mu (t) := \Exp N_t$ and $v(t) := \Var N_t$.
Since $N_t =0$, a.s., for $t \geq 1$, we have $\mu(t) = v(t) = 0$ for all $t \geq 1$,
so of interest is only $t \in (0,1)$.
The  following exact results are known.

\begin{proposition}[Lootgieter, 1977; van Zwet, 1978]
\label{prop:variance}
It holds that
\begin{equation}
\label{eq:mu-t}
\mu (t) = \left( \frac{2}{t} - 1 \right) \1{ 0 < t < 1} .
\end{equation}
Moreover, with $s_0 := 8 \log 2 - 5 \approx 0.545177$, it holds that
\begin{equation}
\label{eq:var-t}
 v(t) = \begin{cases} \displaystyle \frac{s_0}{t} & \text{if } 0 < t \leq 1/2 , \\[1em]
\displaystyle  2 + \frac{2 - 8 \log t}{t} - \frac{4}{t^2} &  \text{if } 1/2 < t < 1 .\end{cases} \end{equation}
\end{proposition}

\begin{remark}
\label{rem:variance}
    Proposition~\ref{prop:variance}
    is due, independently, to van Zwet~\cite{vanzwet} (submitted February~1977) and Lootgieter~\cite{lootgieter-77b} (September 1977).
    Proposition 1.2 of~\cite[p.~396]{lootgieter-77b} covers both results for $\mu$ and $v$, while \cite{vanzwet} has the result for $\mu$, and showed that $v(t) = v(1/2)/(2t)$ for $0< t \leq 1/2$, 
but had not evaluated $v(1/2) =  16 \log 2 - 10$. 
Both proofs go by an analysis of the recursion~\eqref{eq:rde-U}. 
The full result for~$v$ was rediscovered by Pyke \& van~Zwet~\cite[p.~392]{pvz}. 
Furthermore, $v(t)$ coincides with the quantity $v_e (1/t,0)$
in~\cite{si}, which  satisfies the same integral equation, and then~\eqref{eq:var-t} can be found in~\cite[pp.~75, 83]{si}.
\end{remark}

To prepare for our later arguments, we
need estimates for moments of $N_t$ and of
 $M_n$, given in Lemmas~\ref{lem:N-moments} and~\ref{lem:M-moments}, respectively;
 parts of our arguments 
build on some of the analysis in~\cite{pvz}.
The intuition in both cases is that the variables are concentrated about their respective means,
namely $N_t \approx 2/t$ (for small~$t$) and $M_n \approx 2/n$ (for large~$n$). 
The following rough, but useful, upper bounds on the moments of $N_t$ are derived directly from results in~\cite{pvz}.

\begin{lemma}
\label{lem:N-moments}
For each $k \in \N$ there exists $C_k \in \RP$ with  $\Exp ( N_t^k ) \leq C_k t^{-k}$, for all $0 < t \leq 1$.
\end{lemma}
\begin{proof}
Let $k \in \N$.
It follows from~\eqref{eq:nt-def} that 
$\Pr ( N_t = 0) = 1$ for $t \geq 1$, and 
$\Pr ( 1 \leq N_t \leq N_s < \infty ) =1$ for all $0 < s \leq t < 1$,
so that $\Exp (N_t^k)$ is non-increasing in $t > 0$.
Lemma~2.1 of~\cite[p.~385]{pvz} shows that, for every $t_0 > 0$,
\begin{equation}
\label{eq:big-t-moment-bound}
\sup_{t \geq t_0} \Exp ( N_t^k ) = \Exp ( N_{t_0}^k ) < \infty .
\end{equation}
Moreover, Theorem~2.2 of~\cite[p.~386]{pvz} and the algebra
relating cumulants to moments~\cite[pp.~266--7]{moran} shows that there is a constant $C_k \in \RP$
such that 
\begin{equation}
\label{eq:small-t-moment-bound}
\Exp ( N_t^k ) \leq C_k t^{-k}, \text{ for } 0 < t \leq \frac{1}{k} .
\end{equation}
Combining~\eqref{eq:small-t-moment-bound}
with the $t_0 = 1/k$ case of~\eqref{eq:big-t-moment-bound}, we verify the statement in the lemma.
\end{proof}

Next are bounds on the moments of $M_n$.

\begin{lemma}
\label{lem:M-moments}
For each $k \in \N$ there exists $C'_k \in \RP$ with  $\Exp ( M_n^k ) \leq C'_k n^{-k}$, for all $n \in \N$.
\end{lemma}
\begin{proof}
Let $k \in \N$. By Lemma~\ref{lem:N-moments} and Markov's inequality, for $0< t \leq 1$ and $n \in \N$,
\begin{equation}
\label{eq:M-moments-2} \Pr ( N_t \geq n ) \leq \frac{\Exp ( N_t^{k+1} )}{n^{k+1}}
\leq \frac{C_{k+1}}{n^{k+1}} \cdot t^{-k-1}  .\end{equation}
From the  integration by parts formula for moments~\cite[p.~75]{gut}, combined with~\eqref{eq:M-N-inverse} and the fact that
$\Pr ( \frac{1}{n+1} \leq M_n \leq 1) =1$ for all $n \in \N$, we obtain
\begin{equation}\label{eq:M-moments-1}
\Exp ( M_n^k )  = k \int_0^1 t^{k-1} \Pr ( M_n > t ) \ud t 
 \leq k \int_0^{(n+1)^{-1}} t^{k-1} \ud t  + k \int_{(n+1)^{-1}}^{1} t^{k-1} \Pr ( N_t \geq n ) \ud t .
\end{equation}
Hence from~\eqref{eq:M-moments-1} and~\eqref{eq:M-moments-2}, 
for $n \in \N$,
\begin{align*}
\Exp ( M_n^k ) & \leq 
\frac{1}{(n+1)^k} + \frac{k C_{k+1}}{n^{k+1}} \int_{(n+1)^{-1}}^{1}   t^{-2}  
\ud t \leq \bigl( 1 + 2 k C_{k+1} \bigr) n^{-k},
\end{align*}
which yields the claimed bound, with $C'_k := 1 + 2k C_{k+1}$.
\end{proof}

We turn to centred moments of $M_n$;
the intuition here, from the CLT in
Proposition~\ref{prop:M-clt},
is that  $\sqrt{n} ( n M_n - 2 )$
is tight. The precise statement that we need
is the following, which exploits some further ideas from Pyke \& van Zwet~\cite{pvz}.

\begin{lemma}
    \label{lem:M-deviations}
    There is a constant $B_0 < \infty$ such that
\begin{equation}
    \label{eq:nM-claim}
  \sup_{n \in \N} \left( \sqrt{n} \Exp | n M_n - 2| \right) \leq B_0.
\end{equation}   
\end{lemma}
\begin{proof}
Since $(n+1) M_n \geq 1$, a.s., 
it holds that
$nM_n - 2 = (n+1) M_n -2 - M_n$ satisfies
\begin{equation}
    \label{eq:M-bounds}
    -(n+2) M_n \leq -1 -M_n \leq nM_n -2 \leq n M_n , \text{ for all } n \in \N.
    \end{equation}
By~\eqref{eq:M-bounds} we see that 
$| n M_n - 2 | \leq (n+2) M_n$, a.s. Then,
by Markov's inequality and the fact that $\Exp ( M_n^k ) = O( n^{-k})$ from Lemma~\ref{lem:M-moments},
we obtain
\[ \Pr ( | n M_n - 2 | > n^{1/6} ) \leq \frac{(n+2)^{12} \Exp ( M_n^{12})}{n^2} \leq C n^{-2}, \text{ for all } n \in \N.\]
Moreover, since  $| n M_n -2 | \leq n+2$, a.s., it follows that
\begin{equation}
\label{eq:M_n-central-mean-big}
  \sqrt{n} \Exp \bigl(  | n M_n - 2 | \1 { | n M_n - 2 | > n^{1/6} } \bigr)
 \leq (n+2) \sqrt{n} \Pr ( | n M_n - 2 | \geq n^{1/6} ) 
 = O ( n^{-1/2} ), \text{ as } n \to \infty.
\end{equation}

Next we follow~\cite[pp.~400--1]{pvz}. Let   $n \in \N$;
recall that $\Var N_t = v(t)$ and, from~\eqref{eq:mu-t}, that $\Exp N_t = \mu (t) = \frac{2}{t} -1$.
From the inversion relation~\eqref{eq:M-N-inverse} and 
Chebyshev's inequality, 
\[ \Pr (  M_n > t ) = \Pr \left( \frac{N_t - \Exp N_t}{\sqrt{ \Var N_t}} > \frac{n - \mu (t)}{\sqrt{v(t)}} \right)
\leq \frac{v(t)}{(n-\mu(t))^2} , \text{ for } \frac{2}{n+1} < t < 1.
\]
Similarly,
\[ \Pr (  M_n \leq t ) = \Pr \left( \frac{N_t - \Exp N_t}{\sqrt{ \Var N_t}} \leq - \frac{\mu (t) - n}{\sqrt{v(t)}} \right)
\leq \frac{v(t)}{(n-\mu(t))^2} , \text{ for } 0 < t < \frac{2}{n+1}.
\]
In particular, taking $t = 2n^{-1} + n^{-3/2} y$,
which, for $0 \leq y \leq n^{2/3}$ has $t \in (0,1/2)$
for all $n > n_0 := 6^{6/5}$, the formula $v(t) = s_0/t$ (with $s_0=8\log 2 -5$)
from Proposition~\ref{prop:variance} yields
\begin{align}
\label{eq:n-M-bound-1}
\Pr ( n M_n > 2 + n^{-1/2} y ) & \leq \frac{s_0 n}{(2  + n^{-1/2} y)(n+1 - \frac{2n}{2 + n^{-1/2} y} )^2} \nonumber\\
& = \frac{s_0}{(2  + n^{-1/2} y) n (\frac{1}{n} + \frac{n^{-1/2}y}{2 + n^{-1/2} y} )^2} \nonumber\\
& \leq \frac{s_0 (2  + n^{-1/2} y)}{y^2}, \text{ for all } 0 \leq y \leq n^{2/3}.
\end{align}
Similarly, taking $t = 2n^{-1} - n^{-3/2} y$,
for $0 \leq y \leq 2 \sqrt{n}$, we get, for $n > 4$,
\begin{equation*}
\Pr ( n M_n < 2 - n^{-1/2} y ) \leq \frac{s_0 n}{(2  - n^{-1/2} y)(n+1 - \frac{2n}{2 - n^{-1/2} y} )^2} 
 = \frac{s_0}{(2  - n^{-1/2} y) n (\frac{1}{n} - \frac{n^{-1/2}y}{2 - n^{-1/2} y} )^2} .
\end{equation*}
If $y \geq 2/\sqrt{n}$ and $n >4$, then 
$\frac{n^{-1/2}y}{2 - n^{-1/2} y} \geq \frac{1}{n-1} > \frac{1}{n}$,
so we obtain
\begin{equation}
\label{eq:n-M-bound-2}
\Pr ( n M_n < 2 - n^{-1/2} y ) 
\leq  \frac{ s_0 (2  - n^{-1/2} y)}{y^2}, \text{ for all } 2 /\sqrt{n} \leq y \leq 2 \sqrt{n}.
\end{equation}
Summing~\eqref{eq:n-M-bound-1} and~\eqref{eq:n-M-bound-2}, we conclude that, for all $n > n_0$,
\[ \Pr ( | nM_n - 2 | \geq n^{-1/2} y ) \leq \frac{4 s_0 }{y^2}
, \text{ for all } 2 /\sqrt{n} < 1  \leq y \leq n^{2/3}. \]
For a random variable $X \in \RP$ and a constant $a \in (0,\infty)$, we have~(e.g.~\cite[p.~75]{gut})
\[ \Exp ( X \1 { X \leq a} ) = \int_0^\infty \Pr ( X \1{ X \leq a } > y ) \ud y 
\leq \int_0^a \Pr (X > y ) \ud y,\]
which, applied with $X = \sqrt{n} | nM_n -2|$ and $a = n^{2/3}$, yields 
\begin{align*} 
\sup_{n\in \N} \sqrt{n} \Exp \bigl(   | n M_n - 2 |  \1 { | n M_n -2 | \leq n^{1/6} } \bigr) 
&\leq \sup_{n \in \N} \int_0^{n^{2/3}} \Pr ( | nM_n - 2 | > n^{-1/2} y ) \ud y\\ 
&\leq 1 + \int_1^\infty \frac{4s_0}{y^2} \ud y 
< \infty.
\end{align*}
Combining the above bound with~\eqref{eq:M_n-central-mean-big} completes the proof.
\end{proof}

The last result of this section is more technical in nature, concerning moments of harmonic sums of the~$M_n$; it plays an important role in the sections below.

\begin{corollary}
\label{cor:reciprocal-moments}
There is a constant $B_1 < \infty$ such that,
for all $k \in \N$ and all $n \in \N$,
    \begin{equation}
    \label{eq:reciprocal-moments}
  \Exp \left| \frac{4^k}{n^{2k}} \left( \sum_{j=0}^{n-1} \frac{1}{M_j} \right)^k - 1 \right| \leq \frac{B_1 4^k k}{\sqrt{n}} .
\end{equation}
\end{corollary}
\begin{proof}
Define
$W_n := \sum_{j=0}^{n-1} M^{-1}_j$. 
Since $\inf_{0 \leq i \leq n-1} M_i \geq M_{n-1} \geq 1/n$,
note that $W_n \leq n^2$, a.s. 
Let $k \in \ZP$, and observe that, for every $n \in \N$, 
\begin{align*}
\left| W_n^{k+1} - \frac{n^2}{4} W_n^{k}  \right|
& \leq 
\left| W_n^{k+1} - W_n^{k} \sum_{i=0}^{n-1} \frac{i}{2} \right| + \left( \frac{n}{4} \right) W_n^k  \\
& \leq 
\frac{W_n^{k}}{2}  \sum_{i=0}^{n-1}  \frac{| i M_i - 2 |}{M_i} + n W_n^k  \\
& \leq 
 n^{2k+1}  \sum_{i=0}^{n-1}  | i M_i - 2 |  + n^{2k+1} ,
\end{align*}
using the bounds $W_n \leq n^2$ and $M_i \geq 1/n$. 
By Lemma~\ref{lem:M-deviations}, we thus obtain 
\begin{equation}
    \label{eq:Wk-to-Wk+1}
 \Exp \Bigl| W^{k+1}_{n} - \frac{n^2}{4} W_n^k  \Bigr| \leq B_1 n^{2k +(3/2)} , \text{ for all } k, n \in \N, \end{equation}
 where $B_1 := 2 B_0 +1 <\infty$.
Using~\eqref{eq:Wk-to-Wk+1} and the triangle inequality, 
\begin{align*}
     \Exp \Bigl| W^{k+1}_{n} - \left( \frac{n^2}{4} \right)^{k+1}  \Bigr|
     & \leq \Exp \Bigl| W^{k+1}_{n} -\frac{n^2}{4} W_n^k   \Bigr|
     + \frac{n^2}{4} \Exp \Bigl| W^{k}_{n} -\left( \frac{n^2}{4} \right)^{k}   \Bigr| \\
     & \leq  B_1 n^{2k +(3/2)} +   {n^2} \Exp \Bigl| W^{k}_{n} -\left( \frac{n^2}{4} \right)^{k}   \Bigr|, \text{ for all } n \in \N.
\end{align*}
An induction on $k$ using the above relation, and the fact that $W_n^0 = 1$, then shows that 
$\Exp | W_n^{k} -  ( n^2/4  )^k | \leq B_1 k  n^{2k-(1/2)}$, for all $k, n \in \N$. This
yields~\eqref{eq:reciprocal-moments}.
\end{proof}

\section{Small-gap statistics}
\label{sec:small-gaps}

\subsection{Conditional independence structure and moments}
\label{sec:small-gaps-independence}

To obtain Theorem~\ref{thm:small-gap} on the smallest gap, it is not surprising that we investigate the count of small gaps and use Poisson approximation. However, our approach to studying the fine fluctuations of the \emph{largest} gap turns out
to make essential use of more detailed information about  small gaps, and the primary focus of this section is to present this detailed information. Later in this section we will then present the proof of Theorem~\ref{thm:small-gap}, the main ingredient being Corollary~\ref{cor:Knt-representation} that we state shortly.

Of course a typical gap has length about $1/n$, and   $M_n$  is of the same order
(about $2/n$; see~\eqref{eq:slln} and Proposition~\ref{prop:M-clt}); on the other hand, as Theorem~\ref{thm:small-gap} advertises, one expects to see small gaps all the way down to size around $1/n^2$.
The results in this section will give more information on small gaps, including those with lengths~$o (1/n)$.

For $g : [0,1] \to \R$, $n \in \N$, and $0 \leq s < t \leq 1$,    define the statistic 
 \begin{equation}
    \label{eq:U-n-t-def} \cK^g_{n} (s,t] := \sum_{i=1}^{n+1} g \Bigl( \frac{L_{n,i}}{t} \Bigr) \1{s < L_{n,i} \leq t} .
    \end{equation}
In the case where $g \equiv 1$, then $\cK^g_{n}$ is a counting function; we use the particular notation $K_n := \cK^g_n$, in that case. That is, for $n \in \N$ and $0 \leq s < t \leq 1$, 
\begin{equation}
    \label{eq:K-n-t-def}
 K_{n} (s,t] := \sum_{i=1}^{n+1}  \1 {  s < L_{n,i} \leq t }  ,\end{equation}
the number of gaps of size in $(s,t]$.  Write $K_{n,t} := K_{n} (0,t]$.
Some intuition for these quantities is provided by Pyke's uniform limit theorem~\eqref{eq:e-d-uniform}, a consequence of which 
is that, for $g : [0,1] \to \R$  bounded and measurable, and $0 \leq \alpha < \beta \leq 2$,  
    \begin{equation}
\label{eq:U-lln}
\lim_{n \to \infty} n^{-1} \cK^g_n \bigl( \tfrac{\alpha}{n+1},\tfrac{\beta}{n+1} \bigr]
 = \frac{1}{2} \int_{\alpha}^{\beta} g  ( u /\beta ) \ud u, \as \end{equation}
Another consequence of~\eqref{eq:e-d-uniform} is
    \begin{equation}
\label{eq:pyke-K-lln}
\lim_{n \to \infty} \sup_{t \in \bigl[ 0, \tfrac{1}{2(n+1)} \bigr]} \left| \frac{K_{n,t}}{n+1} - \frac{(n+1) t}{2} \right| = 0, \as  
\end{equation}
There are also second-order (fluctuation) results that complement~\eqref{eq:U-lln}--\eqref{eq:pyke-K-lln},
provided by~\cite{pvz}. 
However,
these results are 
targeted at \emph{typical} gaps, and are 
of limited value concerning $\cK^g_n (s,t]$ when $t \ll 1/n$.
Indeed, roughly speaking, the asymptotic~\eqref{eq:pyke-K-lln} says that, 
$K_{n,t} = O ( n^2 t) + o(n)$, a.s., 
but for $t \ll 1/n$ it is the~$o(n)$ term that dominates. 

The aim of this section is to provide a  sharper study of small gaps, which will
allow us to conclude, for example, that
$K_{n,t} = O ( n^2 t)$ even when $n t \to 0$ (see Lemma~\ref{lem:small-gaps-K} below). The additional
structure we need is provided by the following important conditional independence result.
In particular, representation~\eqref{eq:U-representation-K} shows that $K_{n,t}$
can be represented as a sum of a random number of terms involving \emph{independent} $\unif{0}{1}$ random variables.

\begin{lemma}
\label{lem:small-gaps-U}
Suppose that $0 \leq s < t \leq 1$.
There exist random variables $\theta_0, \theta_1, \ldots$ and $\gamma_0, \gamma_1, \ldots$,
such that (i)
$\gamma_0, \gamma_1, \ldots$ are i.i.d.~$\unif{\frac{s}{t}}{1}$ random variables,
independent of the~$\theta_i$;
(ii) given $M_0, M_1, \ldots$, the  
$\theta_0, \theta_1, \ldots \in \{0,1\}$ are independent with
$\Pr ( \theta_i  = 1 \mid M_0, M_1, \ldots ) = 2(t-s) / M_i$;
(iii)
for every $g : [0,1] \to \R$ and every $n \in \ZP$ for which $2nt \leq 1$,
we have the representation 
\begin{equation}
\label{eq:U-representation-theta} 
\cK^g_{n} (s,t] = \sum_{i=0}^{n-1}  \theta_i g (\gamma_i) . \end{equation}
In other words, 
for fixed $n \in \N$ and $0 \leq s < t$ with $2nt \leq 1$, we can write
\begin{equation}
\label{eq:U-representation-K} 
\cK^g_{n} (s,t] = \sum_{j=1}^{K_{n}(s,t]}  g (\upsilon_j) , \end{equation}
where $\upsilon_1, \upsilon_2, \ldots$ are i.i.d.~$\unif{\frac{s}{t}}{1}$, independent of~$K_{n}(s,t]$.
\end{lemma}
\begin{proof}
Fix $n \in \N$ and $t >0$ with $2 n t \leq 1$.
Since $2(n+1) t \leq 1$ and $(i+1) M_i > 1$, a.s., we have $M_i > 2t$ for all $i \in \{0,1,2,\ldots, n-1\}$. 
Hence splitting the interval of length $M_i$ ($0 \leq i \leq n-1$)
can never remove a gap of length in $(s,t]$, and
can create precisely
zero or one gap of length in $(s,t]$, and hence increase $\cK^g$ according to
\begin{multline*}
\cK^g_{i+1} (s,t] - \cK^g_{i} (s,t]\\ 
 = g ( U_{i+1} M_i/t )  \1 { U_{i+1}M_i \in (s, t] }  
+
g ( (1-U_{i+1} ) M_i /t ) \1 { U_{i+1}M_i \in [1-t, 1-s) } ,
\end{multline*}
recalling from~\eqref{eq:splitting-recursion}
that $U_{i+1} \sim \unif{0}{1}$ is the relative location of the
split point in the maximal interval.  Note that, since $2 n t \leq 1$, we have $s <t < 1/2$,
so that intervals $(s,t]$ and $[1-t,1-s)$ are disjoint, each of length $t-s$.
Moreover, conditional on $ U_{i+1}M_i \in (s, t]$, $U_{i+1}M_i/t$ has the
$\unif{\frac{s}{t}}{1}$ distribution;
similarly for $(1-U_{i+1} ) M_i /t$ given $U_{i+1}M_i \in [1-t, 1-s)$.
Thus we obtain the claimed representation~\eqref{eq:U-representation-theta} on setting
\begin{align*}
    ( \theta_i , \gamma_i  ) & := (1, U_{i+1} M_i/t ) \1 { U_{i+1}M_i \in (s, t] }
        +  (1, (1-U_{i+1} ) M_i /t )\1 { U_{i+1}M_i \in [1-t, 1-s) } \\
       & {} \qquad {}     +  (0, V_i )\1 { U_{i+1}M_i \notin (s,t] \cup [1-t, 1-s) },
\end{align*}
for $V_1, V_2, \ldots$ a sequence of i.i.d.~$\unif{\frac{s}{t}}{1}$ random variables (merely to ensure that $\gamma_i$
has the correct distribution even if $\theta_i = 0$).

The second expression~\eqref{eq:U-representation-K} is obtained by ignoring the terms in~\eqref{eq:U-representation-theta} where $\theta_i =0$ and re-labelling so that $\upsilon_j := \gamma_{h(j)}$ where $h(j) := \inf \{ i \in \ZP : \sum_{k=0}^i \theta_k = j\}$. 
The number of non-zero terms is exactly $K_n (s,t] = \sum_{i=0}^{n-1} \theta_i$, and the independence
structure in~\eqref{eq:U-representation-K} means that $K_n (s,t]$ is independent of the $\upsilon_j$
in~\eqref{eq:U-representation-K}.
\end{proof}

Taking $g \equiv 1$ in~\eqref{eq:U-representation-theta}
gives the following useful fact.

\begin{corollary}
    \label{cor:Knt-representation}
Suppose that $n \in \N$ and $t >0$ satisfy $2n t \leq 1$. Then 
 \begin{equation}
 \label{eq:Knt-representation} 
 K_{n,t}   = \sum_{i=0}^{n-1}  \theta_i , \text{ for } n \in \N, \end{equation}
 where, given $M_0, M_1, M_2, \ldots$, the 
 $\theta_i \in \{0,1\}$ are independent with 
 $\Pr ( \theta_i = 1 \mid M_i ) = 2t / M_i$.     
\end{corollary}

Corollary~\ref{cor:Knt-representation} is the basis for the proof of Theorem~\ref{thm:small-gap}, which uses Poisson approximation and is presented later in this section. 
For Theorem~\ref{thm:clt} we need to further develop analysis of $K_{n,t}$. 
The following result gives asymptotics for the moments of $K_{n,t}$,
which will be a key ingredient in the subsequent arguments. Part~\ref{lem:small-gaps-K-i} gives an upper bound valid for a broad range of the parameters,
  part~\ref{lem:small-gaps-K-ii} gives sharp asymptotics for a more restrictive
range of parameters, and part~\ref{lem:small-gaps-K-iii} gives a tail bound.

\begin{lemma}
\label{lem:small-gaps-K}
Suppose that $n \in \N$. Then the following hold.
\begin{thmenumi}[label=(\roman*)]
\item
\label{lem:small-gaps-K-i}
For every $t  >0$, $k \in \ZP$, and $n \in \N$ with 
$2n t \leq 1$,
\begin{equation}
\label{eq:Knt-moments-crude}
 \Exp \left( K_{n,t}^k \right) 
 \leq ( 2 n^2 t)^k \exp \left( \frac{k^2}{4n^2 t} \right).
\end{equation}
\item \label{lem:small-gaps-K-ii}
Let $\nu \in (1,\frac{3}{2})$. There exist constants
$\delta >0$ and
$B_2 < \infty$ such that, for all $n, k \in \N$, 
\begin{equation}
\label{eq:exp-Knt-asymptotics}
\sup_{t \in \bigl[ n^{-\nu}, \frac{1}{2n} \bigr]} \left| \frac{2^k}{ t^k n^{2k}} \Exp  \bigl( K^k_{n,t} \bigr) - 1 \right| 
\leq \frac{B_2 4^k k}{\sqrt{n}} , \text{ whenever } k \leq \delta \log n. \end{equation}
\item
\label{lem:small-gaps-K-iii}
For every $t >0$ with $2n t \leq 1$, 
we have $\Pr \left(  K_{n,t} \geq 6 t n^2 \right) \leq \exp  \left( - 2 t^2 n^3 \right)$.
\end{thmenumi}
\end{lemma}
\begin{proof}  
Recall the representation for $K_{n,t}$ from Corollary~\ref{cor:Knt-representation}, and  that  $\Pr ( \theta_i = 1 \mid M_i ) = 2t / M_i \leq 2 n t$
 for $i \in \{0,1,\ldots,n-1\}$, since
  $M_i > \frac{1}{i+1}$, a.s.
  Consequently, the moment generating function of~$K_{n,t}$ is dominated by that of a $\Bin{n}{2n t}$ random variable.
  Hence (see~\cite[\S 3]{ahle}) we may apply the tail bound from Theorem~1 of~\cite{ahle}, which yields
part~\ref{lem:small-gaps-K-i}.

 Next we
 prove part~\ref{lem:small-gaps-K-ii}.
For $k \in \N$
let $I_{n,k} := \{0,1,\ldots, n-1\}^k$.
By~\eqref{eq:Knt-representation},
\begin{equation}
    \label{eq:Knt-expanded-moments}
 \Exp \left( K_{n,t}^k \right)
=  \sum_{(i_1, \ldots, i_k) \in I_{n,k} }  \Exp \left( \theta_{i_1} \cdots \theta_{i_k}  \right)
.
\end{equation}
Let $I^\circ_{n,k} \subset I_{n,k}$ be the set of all $(i_1, \ldots, i_k ) \in I_{n,k}$ 
for which the~$k$ coordinates are distinct.
Then $I_{n,k}$ contains $n^k$
elements, while $I^\circ_{n,k}$ contains 
$\frac{n!}{(n-k)!}$ elements.
For $k \geq 2$, every element of $I_{n,k} \setminus I^\circ_{n,k}$
contains at least one pair of the~$k$ coordinates that
match, so
\begin{equation}
    \label{eq:I-n-k-bound}
    | I_{n,k} \setminus I^\circ_{n,k} | = n^k - \frac{n!}{(n-k)!} \leq \binom{k}{2} | I_{n,k-1} |
    \leq \frac{k^2}{2} n^{k-1}  ;
\end{equation}
when $k=1$, clearly $I_{n,1} \setminus I^\circ_{n,1} =\emptyset$.
Now, from~\eqref{eq:Knt-expanded-moments}, 
\begin{equation}
\label{eq:I-split}
\Exp \left( K_{n,t}^k \right)
=  \sum_{(i_1, \ldots, i_k) \in I^\circ_{n,k} } \Exp \left( \theta_{i_1} \cdots \theta_{i_k}  \right)
+   \sum_{(i_1, \ldots, i_k) \in I_{n,k} \setminus I^\circ_{n,k}} \Exp \left(  \theta_{i_1} \cdots \theta_{i_k}  \right)
.
\end{equation}
For the first term on the right-hand side of~\eqref{eq:I-split}, by conditional independence,
\begin{equation}
    \label{eq:conditional-circ}
\Exp \left( \sum_{(i_1, \ldots, i_k) \in I^\circ_{n,k}} \theta_{i_1} \cdots \theta_{i_k} \mid M_0, M_1, \ldots \right)
=  (2t)^k \sum_{(i_1, \ldots, i_k) \in I^\circ_{n,k}}  {M_{i_1}^{-1}} \cdots {M_{i_k}^{-1}} .
\end{equation}
We bound the error between the sum on the right of~\eqref{eq:conditional-circ}  and the quantity
\[ \Exp \biggl( \Bigl( \sum_{i=0}^{n-1}  {M_i^{-1}} \Bigr)^k \biggr) = \Exp \sum_{(i_1, \ldots, i_k) \in I_{n,k}} {M_{i_1}^{-1}} \cdots  {M_{i_k}^{-1}} \]
from Corollary~\ref{cor:reciprocal-moments}. Indeed,  using~\eqref{eq:I-n-k-bound} and the fact that $M_i > 1/n$ for all $0 \leq i \leq n-1$,
\begin{align*}
{} & {} \biggl| \sum_{(i_1, \ldots, i_k) \in I_{n,k}}  {M_{i_1}^{-1}} \cdots  {M_{i_k}^{-1}} - \sum_{(i_1, \ldots, i_k) \in I^\circ_{n,k}}  {M_{i_1}^{-1}} \cdots  {M_{i_k}^{-1}} \biggr| \\
& {} \quad {} = 
\biggl| \sum_{(i_1, \ldots, i_k) \in I_{n,k}\setminus I^\circ_{n,k}}  {M_{i_1}^{-1}} \cdots  {M_{i_k}^{-1}} \biggr| \leq \frac{k^2}{2} n^{2k-1} , \text{ for all } k, n \in \N .\end{align*}
Thus, taking expectations in~\eqref{eq:conditional-circ}, we obtain
\[ 
\biggl|  \sum_{(i_1, \ldots, i_k) \in I^\circ_{n,k}} \Exp  \left( \theta_{i_1} \cdots \theta_{i_k} \right) - (2t)^k \Exp \biggl( \Bigl( \sum_{i=0}^{n-1}  {M_i^{-1}} \Bigr)^k \biggr)   \biggr| 
\leq 2^k t^k k^2 n^{2k-1} , \text{ for all } k, n \in \N .
\]
Then
from  Corollary~\ref{cor:reciprocal-moments}, there is a  $C < \infty$ (depending on $B_1$) such that, for all $n, k \in \N$,
\begin{align*}
\biggl|  \sum_{(i_1, \ldots, i_k) \in I^\circ_{n,k}} \Exp  \left( \theta_{i_1} \cdots \theta_{i_k} \right) - \frac{n^{2k}t^k}{2^k}   \biggr| 
& \leq C \left( 4^k k  +  4^k k^2 n^{-1/2} \right) \frac{n^{2k} t^k}{2^k} n^{-1/2} . \end{align*}
Provided $k \leq n^{1/2}$ we have $k n^{-1/2} \leq 1$,
so we conclude that, for a constant $C < \infty$, for all $n \in \N$, 
\begin{align}
\label{eq:k-to-n}
\biggl|  \sum_{(i_1, \ldots, i_k) \in I^\circ_{n,k}} \Exp  \left( \theta_{i_1} \cdots \theta_{i_k} \right) - \frac{n^{2k}t^k}{2^k}   \biggr| 
& \leq C 4^k k \frac{n^{2k} t^k}{2^k} n^{-1/2} , \text{ for all } k \leq n^{1/2}.
\end{align}
Similarly to~\eqref{eq:I-n-k-bound}, using the fact that the $\theta_i$ are $\{0,1\}$-valued,
\begin{align*}
   \sum_{(i_1, \ldots, i_k) \in I_{n,k} \setminus I^\circ_{n,k}} \Exp \left( \theta_{i_1} \cdots \theta_{i_k}  \right) & \leq \frac{k^2}{2}
   \sum_{(i_1, \ldots, i_{k-1}) \in I_{n,k-1}}  \Exp \left( \theta_{i_1} \cdots \theta_{i_{k-1}}  \right)  \\
   & \leq \frac{k^2}{2} \Exp \left( K_{n,t}^{k-1} \right) \leq 2^{k-2} k^2 (n^2 t)^{k-1} \exp \left( \frac{k^2}{4 n^2 t} \right),
  \end{align*}
using the $k-1$ cases of~\eqref{eq:Knt-expanded-moments} and~\eqref{eq:Knt-moments-crude}. 
Fix $\delta >0$.
Since $t \geq n^{-\nu}$
for $\nu < 3/2$, and $k \leq \delta \log n$,
we have that $k^2 /(n^2 t)$ is uniformly bounded. Hence, 
combining the preceding display and~\eqref{eq:k-to-n} with~\eqref{eq:I-split},
we obtain, for some $C < \infty$ and all $n, k \in \N$ with $k \leq \delta \log n$,
\begin{align*}
    \left| \Exp \bigl( K_{n,t}^k \bigr)  - \frac{n^{2k}t^k}{2^k} \right| 
& \leq C \frac{n^{2k} t^k}{2^k} \left( 4^k k n^{-1/2}  + 4^k k^2 (n^2 t)^{-1} \right), \text{ for all }
t \in \left[ n^{-\nu} , \tfrac{1}{2n}\right].
\end{align*}
Here $n^{1/2} k (n^2 t)^{-1} \leq \delta n^{\nu-(3/2)} \log n \to 0$
(recall that $1 < \nu < 3/2$ and $t \geq n^{-\nu}$)
so we obtain~\eqref{eq:exp-Knt-asymptotics},
completing the proof of~\ref{lem:small-gaps-K-ii}.
 
Finally, we prove~\ref{lem:small-gaps-K-iii}.
Let $\cF_n := \sigma (U_1,\ldots, U_n)$ be the
$\sigma$-algebra generated by the first $n$~divisions;
 $\cF_0$ is the trivial $\sigma$-algebra. 
Fix $n \in \N$ with $2nt \leq 1$.
Note 
$\Exp ( K_{m+1,t}- K_{m,t} \mid \cF_m ) = \Exp ( \theta_{m} \mid \cF_m ) = 2t / M_m$
by~\eqref{eq:Knt-representation}. 
Let $A_{0,t} := 0$ and, for $m \in \N$,
$A_{m,t} := 2t\sum_{j=0}^{m-1} 1/M_j   ,$
and set $X_{m,t} := K_{m,t} - A_{m,t}$ for $m \in \ZP$. 
Then $X_{m,t}$ is $\cF_m$-measurable,  $A_{m+1,t} -A_{m,t} = 2t/M_m$, and, provided $m \leq n$,
\[ \Exp ( X_{m+1,t} - X_{m,t} \mid \cF_m ) = \Exp ( K_{m+1,t} -  K_{m,t} \mid \cF_m ) - (A_{m+1,t} -A_{m,t} ) = 0 . \]
Thus $X_{0,t}, X_{1,t}, \ldots, X_{n,t}$ is a martingale. Moreover, $0 \leq K_{m+1,t} - K_{m,t} \leq 1$ and 
$0 \leq A_{m+1} - A_{m,t} \leq  2 (m+1) t \leq 1$ for $m \leq n-1$, and so
$\sup_{0 \leq m \leq n-1} |X_{m+1,t}-X_{m,t}| \leq 1$, a.s. We    apply a one-sided  Azuma--Hoeffding inequality
(see e.g.~\cite[p.~46]{mpw}) to obtain, for all $ a \in \RP$, 
$\Pr ( |X_{n,t} - X_{0,t} | \geq a ) \leq  \exp ( - a^2 / (2n) )$. 
Since $A_{n,t} \leq  2t\sum_{i=0}^{n-1} (i+1)  \leq 4 t n^2$, a.s.,  
\[ \Pr ( K_{n,t} \geq 6 t n^2 ) \leq
\Pr ( K_{n,t} \geq 2 t n^2 + A_{n,t} ) \leq   \exp \left( - \frac{4 t^2 n^4}{2n} \right) ,
\]
which yields the tail bound in part~\ref{lem:small-gaps-K-iii}.
\end{proof}

\subsection{Limit theorem for the smallest gap}
\label{sec:smallest-gap}

In this section we use some standard Poisson approximation bounds, the representation 
given in Corollary~\ref{cor:Knt-representation} for counts~$K_{n,t}$ of small gaps, defined at~\eqref{eq:K-n-t-def}, and the reciprocal moments bounds in Corollary~\ref{cor:reciprocal-moments}, to give a proof of  Theorem~\ref{thm:small-gap} on the asymptotics of the smallest gap, $m_n = \min_{1 \leq i \leq n+1} L_{n,i}$ defined at~\eqref{eq:M-def}.

Recall from Corollary~\ref{cor:Knt-representation} that $K_{n,t}=\sum_{i=0}^{n-1}\theta_i$, where the $\theta_i$ are supported on $\{0,1\}$, are conditionally independent given $M_0,M_1,\ldots$, and satisfy $\Pr(\theta_i=1 \mid M_0,M_1,\ldots)= 2t/M_i$.
Throughout this section we let $C$ denote a positive, finite constant which is independent of $n$ and $t$ and whose value may vary from line to line.

\begin{proof}[Proof of Theorem~\ref{thm:small-gap}]
For non-negative, integer-valued random variables $K$ and $Y$, the corresponding total variation distance is denoted by
\[
\dtv(K,Y)=\sup_{A\subseteq\mathbb{Z}_+}|\mathbb{P}(K\in A)-\mathbb{P}(Y\in A)|.
\]
Recall a classic bound of Le Cam~\cite{lecam} (see also~\cite[p.~3]{bhj92}): letting $I_1,\ldots,I_n$ be independent Bernoulli random variables with $\Exp I_i =p_i$, the total variation distance (denoted below by $\dtv$) between $\sum_{i=1}^n I_i$ and a $\text{Pois} (\sum_{i=1}^n p_i )$ random variable is bounded by $4.5\max_{1\leq i\leq n}p_i$. Let $Y\sim\text{MP}( 2t \sum_{i=0}^{n-1} M_i^{-1} )$ have a mixed Poisson distribution; that is, conditional on $2t \sum_{i=0}^{n-1} M_i^{-1}$, the random variable $Y$ has a Poisson distribution with this parameter. A conditioning argument combined with Le Cam's result says that 
\[ \dtv
 (K_{n,t},Y) \leq 9t\Exp  \max_{0\leq i\leq n-1}M_i^{-1}  \leq 9 n t ,
\]
using the fact that $M_i > \frac{1}{1+i}$, a.s., for all $i \in \ZP$.
We may now approximate $Y$ by $Z\sim\text{Pois}(\frac{1}{2}n^2t)$. By Theorem 1.C(i) of \cite{bhj92} we have that
\[
\dtv  (Y,Z)\leq\min\left\{1,\sqrt{\frac{2}{n^2t}}\right\} \Exp \left|\sum_{i=0}^{n-1}\frac{2t}{M_i}-\frac{n^2t}{2}\right|\leq Cn^{3/2}t,
\]
for some $C$, where the final inequality follows from Corollary \ref{cor:reciprocal-moments}. Hence, by the triangle inequality there exists $C$ such that
\[
d_\text{TV}(K_{n,t},Z)\leq C\left(nt+n^{3/2}t\right)\leq Cn^{3/2}t.
\]
Then, we choose $t=\frac{2\theta}{n^2}$ for some $\theta>0$ and note that
\[
\Pr\left(m_n>\frac{2\theta}{n^2}\right)=\Pr(m_n>t)=\Pr(K_{n,t}=0)
\]
to obtain that there exists $C$ such that
\begin{equation}\label{eq:exponential}
\left|\Pr\left(m_n>\frac{2\theta}{n^2}\right)-\re^{-\theta}\right|\leq C\theta n^{-1/2},
\end{equation}
which immediately gives us that 
\begin{equation}\label{eq:exponential2}
\left|\Pr\left(\frac{n^2m_n}{2}>x\right)-\re^{-x}\right|\leq\frac{C(1+\log n)}{\sqrt{n}}
\end{equation} 
for any $x\leq1+\log n$. For $x>1+\log n$ we write
\begin{equation*}
\left|\Pr\left(\frac{n^2m_n}{2}>x\right)-\re^{-x}\right|
\leq\max\left\{\Pr\left(\frac{n^2m_n}{2}>x\right),\re^{-x}\right\}
\leq\max\left\{\Pr\left(\frac{n^2m_n}{2}>1+\log n\right),\frac{1}{\re n}\right\}.
\end{equation*}
By \eqref{eq:exponential}, the first term in this final maximum is at most $\frac{C(1+\log n)}{\sqrt{n}}+\frac{1}{\re n}$, and thus~\eqref{eq:exponential2} also holds for these values of $x$ and for a suitable choice of $C$.
\end{proof}

\section{Conditional Berry--Esseen bounds}
\label{sec:berry-esseen}

The starting point of our proof of Theorem~\ref{thm:n-b-e} is  a  decomposition of $N_t$
into a sum of independent, self-similar contributions,
obtained by considering the evolution of the process subsequent to time~$n$.
Fix $n \in \ZP$ and $t \in (0, \frac{1}{n+1})$.
Since $\Pr ( M_n \geq \frac{1}{n+1} > t) = 1$,  $\Pr (N_t > n) =1$. Extending~\eqref{eq:rde-U1} gives the representation,
for $(N^{(i)}_t)_{t>0}$ independent copies of $(N_t)_{t >0}$, independent of
gap lengths $L_{n,1}, \ldots, L_{n,n+1}$ (recall that $\sum_{i=1}^{n+1} L_{n,i} =1$),
\begin{equation}
\label{eq:big-sum}
 N_t - n = Y_{n,t} := \sum_{i=1}^{n+1} N^{(i)}_{t/L_{n,i}} ; \end{equation}
see e.g.~Proposition 1.1 of~\cite{lootgieter-77a} or~\cite{lootgieter-77b}. As a starting-point for proving (non-quantitative) CLTs, 
there is some similarity between~\eqref{eq:big-sum} and the approach of Dvoretzky \& Robbins~\cite{dr} in their
proof of the CLT for R\'enyi's parking model (see also Section~\ref{sec:branching}).

Recall that $\cF_n = \sigma ( U_1, \ldots, U_n)$ defines the filtration  to which the Kakutani process is adapted.  
For $n \in \ZP$ and $t \in (0, \frac{1}{n+1})$, define
\begin{align}
\label{eq:R-def}
 R_{n,t} & := \Exp ( N_t \mid \cF_n) - \Exp N_t = \Exp ( Y_{n,t} \mid \cF_n ) - \Exp Y_{n,t} ; \\
\label{eq:V-def}
 V_{n,t} & := \Var ( N_t \mid \cF_n ) =  \Var ( Y_{n,t} \mid \cF_n ) .\end{align}
 We will
 use the classical Berry--Esseen theorem 
 to obtain the following conditional Berry--Esseen estimate;
 note that in~\eqref{eq:b-e-conditional} not only
 is the probability conditional on~$\cF_n$, but
 so are the centering and scaling quantities $R_{n,t}$ and $V_{n,t}$.

\begin{lemma}
\label{lem:b-e-conditional}
There is a constant $C \in \RP$ such that, for all $n \in \N$ and all $t \in (0, \frac{1}{4(n+1)})$,  
\begin{equation}
    \label{eq:b-e-conditional}
 \sup_{x \in \R} \left| \Pr \left( \frac{N_t-\Exp N_t - R_{n,t}}{\sqrt{ V_{n,t} }} \leq x  \biggmid \cF_n \right) - \Phi (x) \right| \leq 
\frac{C M_n^2}{t^{3/2}}   , \as
\end{equation}
\end{lemma}
\begin{proof}
Fix $n \in\ZP$ and $t \in (0,\frac{1}{4(n+1)})$.
Conditional on $\cF_n$, the summands in the expression given in~\eqref{eq:big-sum} for $Y_{n,t}$ are independent
(although not identically distributed).
Denoting
\[ 
\gamma_{n,t} (i) := \Exp \Bigl( \bigl| N_{t/L_{n,i}} - \mu ( t/L_{n,i} ) \bigr|^3 \Bigmid \cF_n \Bigr),\]
the 
Berry--Esseen theorem 
for sums of independent random variables with finite third moments
(see Theorem~7.6.2 of~\cite[p.~356]{gut})
 yields, for an absolute constant $C \in \RP$,   
 \begin{align}
 \label{eq:conditional-BE-moments}
     \sup_{x \in \R} \left| \Pr \left( \frac{Y_{n,t} - \Exp ( Y_{n,t} \mid \cF_n)}{\sqrt{ \Var ( Y_{n,t} \mid \cF_n)}} \leq x \biggmid \cF_n \right) - \Phi (x) \right| \leq \frac{C \sum_{i=1}^{n+1} \gamma_{n,t} (i)}{\left( \sum_{i=1}^{n+1} v (t/L_{n,i}) \right)^{{3/2}}}, \as
 \end{align}
 Using the elementary inequality $| a- b|^3 \leq  |a|^3 + |b|^3$, $a, b \in \RP$, we have
 \[ \gamma_{n,t}(i) \leq \Exp ( |N_{t/L_{n,i}} |^3 \mid \cF_{n} )
 + 8 t^{-3} L_{n,i}^3 
 \leq C t^{-3} L_{n,i}^3, \as, 
 \]
 for constant $C< \infty$, from~\eqref{eq:mu-t} and~Lemma~\ref{lem:N-moments}.
 Since $\sum_{i=1}^{n+1} L_{n,i} = 1$, it follows that
 \[ \sum_{i=1}^{n+1} \gamma_{n,t} (i) \leq C t^{-3} \Bigl(
 \max_{1 \leq j \leq n+1} L_{n,j}^2  \Bigr) 
 \sum_{i=1}^{n+1} L_{n,i} = C t^{-3} M_n^2. \]
 On the other hand, by~\eqref{eq:var-t},
 provided that $t \in (0, \frac{1}{4(n+1)} )$,
 \begin{equation}
     \label{eq:v-t-L-bound}
 \sum_{i=1}^{n+1} v (t/L_{n,i}) \geq s_0 t^{-1} \sum_{i=1}^{n+1} L_{n,i} \1{ L_{n,i} \geq 2 t }
 \geq \frac{s_0}{t} - 2 s_0 \sum_{i=1}^{n+1} \1{ L_{n,i} < 2 t }
 \geq \frac{s_0}{2t}.
 \end{equation}
Using~\eqref{eq:v-t-L-bound} and the preceding bound
 for $\gamma_{n,t} (i)$  in~\eqref{eq:conditional-BE-moments} yields~\eqref{eq:b-e-conditional}. 
\end{proof}

To deduce
Theorem~\ref{thm:n-b-e}
starting from Lemma~\ref{lem:b-e-conditional},
we need to examine the quantities $R_{n,t}$ and $V_{n,t}$
that appear as centering and scaling in~\eqref{eq:b-e-conditional}. 
To do so, we define
 \begin{equation}
\label{eq:S-def}
 S_{n,t} := v(t) - \Var (N_t \mid \cF_n ) = \Var (Y_{n,t} ) -  V_{n,t} ,\end{equation}
where $V_{n,t}$ is defined at~\eqref{eq:V-def}
and $v (t) = \Var N_t$ is given by~\eqref{eq:var-t}.
A significant part of the remaining technical work of the paper is to obtain good asymptotic estimates for mixed moments of $R_{n,t}$ and $S_{n,t}$ (see Section~\ref{sec:RS-moments}). To facilitate this we derive, in the rest of the present section, basic properties of $R_{n,t}$ and $S_{n,t}$, and crucial representations for $R_{n,t}$ and $S_{n,t}$ in terms of small-gap statistics as described in Section~\ref{sec:small-gaps}. 
 
\begin{lemma}
\label{lem:R-S-algebra}
Suppose that $n \in \ZP$ and $t \in (0, \frac{1}{n+1})$,
and define $R_{n,t}$ and $S_{n,t}$ by~\eqref{eq:R-def} and~\eqref{eq:S-def}.
Then $\Exp R_{n,t} =0$, and the following hold:
\begin{align}
\label{eq:Rvar-R-n-k} \Var ( R_{n,t} \mid \cF_k ) & = \Exp ( S_{n,t} \mid \cF_k ) - S_{k,t}, \text{ for all } k \in \{0,1,2,\ldots, n\}; \\
\label{eq:R-S}
 \Exp S_{n,t} & = \Var R_{n,t} . \end{align}
\end{lemma}
\begin{proof}
 Clearly, $\Exp R_{n,t} = 0$ by~\eqref{eq:R-def}.
Since $Y_{n,t} = N_t - n$ by~\eqref{eq:big-sum}, for $k \leq n$, $\Var ( Y_{n,t} \mid \cF_k ) = \Var (N_t \mid \cF_k)$, 
 and hence, by~\eqref{eq:S-def} and the fact that $\Var Y_{n,t} = v(t)$,
\begin{equation}
\label{eq:R-S-1}
 \Var (N_t \mid \cF_k) = \Var ( Y_{n,t} \mid \cF_k ) = v(t) - S_{k,t}, \text{ for all } k \in \{0,1,\ldots, n \} .\end{equation}
 By the (conditional) total variance formula, using~\eqref{eq:R-def} and~\eqref{eq:S-def}, 
\begin{align*}
\Var (Y_{n,t} \mid \cF_k )  &= \Exp  \bigl( \Var ( Y_{n,t} \mid \cF_n) \bigmid \cF_k \bigr)
+ \Var \bigl(  \Exp ( Y_{n,t} \mid \cF_n ) \bigmid \cF_k \bigr) \\
&= v(t) - \Exp ( S_{n,t} \mid \cF_k ) + \Var ( R_{n,t} \mid \cF_k ).
\end{align*}
 Comparison with~\eqref{eq:R-S-1}
yields~\eqref{eq:Rvar-R-n-k}. Finally, the $k=0$ case of~\eqref{eq:Rvar-R-n-k}  yields~\eqref{eq:R-S}.
\end{proof}

Recall from~\eqref{eq:var-t} 
that $v(t) = \Var N_t = s_0/t$, $t \in (0,1/2)$, where  $s_0 = 8 \log 2 -5$. Set
\begin{equation}
\label{eq:w-def}
 w (t) := \left( 2 + \frac{7-8 \log 2}{t} - \frac{ 8 \log t}{t} - \frac{4}{t^2} \right) \1 { 1/2 < t < 1 } .\end{equation}
The next result includes a representation for $S_{n,t}$ via two sum statistics of the form~\eqref{eq:U-n-t-def}.

\begin{lemma}
\label{lem:S-bound}
Suppose that $n \in \N$ and $t \in (0, \frac{1}{n+1} )$. Then,
with  $w$ defined at~\eqref{eq:w-def},
\begin{equation}
    \label{eq:Snt-representation}
    S_{n,t} = 
     \frac{s_0}{t} \sum_{i=1}^{n+1} L_{n,i} \1 {L_{n,i} \leq t} 
- \sum_{i=1}^{n+1} w ( t/L_{n,i} ).
\end{equation}
Moreover, with $K_{n,t}$ defined at~\eqref{eq:K-n-t-def}, 
whenever $t \in (0, \frac{1}{4(n+1)} )$
 it holds that
\begin{equation}
    \label{eq:S-K-bound}
    S_{n,t} \leq \frac{s_0}{2t}, \text{ and } 
| S_{n,t} | \leq s_0 K_{n,2t}, \as
\end{equation}
\end{lemma}
\begin{proof}
From  Proposition~\ref{prop:variance} and the definition
of $w$ from~\eqref{eq:w-def}, we see that
\begin{equation}
\label{eq:w-appears}
 v (t) - \frac{s_0}{t} \1 { 0 < t < 1} = w (t)  .\end{equation}
The function $w$ as defined in~\eqref{eq:w-def} satisfies $\sup_{0 \leq t \leq 1} |w(t)|  = \lim_{t\to1-} | w(t)| = s_0$, and so
\begin{equation}
\label{eq:v-w-bound}
 | w(t) | =   \left| v(t) - \frac{s_0}{t} \1 { 0 < t < 1} \right| \leq s_0  \1 { 1/2 < t < 1} . \end{equation}
Now let $t \in (0, \frac{1}{n+1})$. 
We have from~\eqref{eq:big-sum} and conditional independence that
\[ \Var ( Y_{n,t} \mid \cF_n ) = \sum_{i=1}^{n+1} \Var ( N_{t/L_{n,i}}^{(i)} \mid \cF_n )
= \sum_{i=1}^{n+1} v ( t/L_{n,i} ) . \]
Hence, from~\eqref{eq:S-def} and~\eqref{eq:w-appears}, since $t < \frac{1}{n+1} \leq 1/2$,
\begin{align}
S_{n,t}  = \frac{s_0}{t} - \sum_{i=1}^{n+1}  v ( t/L_{n,i} ) \label{eq:S-first-formula} = \frac{s_0}{t} - \frac{s_0}{t} \sum_{i=1}^{n+1} L_{n,i} \1 {L_{n,i} > t}  
- \sum_{i=1}^{n+1} w ( t/L_{n,i} ) ,   \end{align}
which yields~\eqref{eq:Snt-representation}. 
If also $t < \frac{1}{4(n+1)}$,  we may apply~\eqref{eq:v-t-L-bound} in~\eqref{eq:S-first-formula} to obtain $S_{n,t} \leq s_0/(2t)$, giving the first bound in~\eqref{eq:S-K-bound}. By~\eqref{eq:Snt-representation}, \eqref{eq:v-w-bound}, and~\eqref{eq:K-n-t-def} we get 
$| S_{n,t} | \leq s_0 K_{n} (0,t] + s_0 K_{n} (t,2t] = s_0 K_{n,2t}$, which gives the second bound in~\eqref{eq:S-K-bound}. 
\end{proof}

\begin{remark}
\label{rem:S-bound}
Consider $t = \theta/n$, so $n^2 t = n \theta$.
From~\eqref{eq:U-lln}, it follows that, for $\theta\in (0,1)$, 
\begin{align*} \lim_{n \to \infty} \frac{1}{n \theta} \sum_{i=1}^{n+1} w  \left( \frac{\theta}{(n+1) L_{n,i}} \right )
& =
\frac{1}{2\theta} \int_\theta^{2\theta} w (1/u) \ud u 
= \frac{1}{2}  \int_{1/2}^{1} t^{-2} w ( t ) \ud t,
\end{align*}
which, using the formula from~\eqref{eq:w-def} to evaluate the integral, takes the (negative) value
\[  \frac{1}{2} \int_{1/2}^{1} t^{-2} w ( t ) \ud t =  2 \log 2 - \frac{17}{12} \approx -0.0303723. \]
Thus~\eqref{eq:Snt-representation}
says that we should expect $S_{n,t}$ to be genuinely of order $n^2 t$.
\end{remark}
 
Next, we show that $R_{n,t}$ can be represented as a sum statistic of the form~\eqref{eq:U-n-t-def}.

\begin{lemma}
\label{lem:R-alternative}
Let $n \in \N$ and take $t \in (0, 1)$.
 Then $R_{n,t}$ defined by~\eqref{eq:R-def} satisfies
\begin{equation}
\label{eq:Rnt-alternative}
  R_{n,t} = \cK_n^g (0,t], \text{ and } | R_{n,t} | \leq K_{n,t}, \as,
 \end{equation}
for $\cK^g$, $g :[0,1] \to [-1,1]$ given by~\eqref{eq:U-n-t-def} with $g(u) = 1-2u$, and $K_{n,t}$  defined at~\eqref{eq:K-n-t-def}.
\end{lemma}
\begin{proof}
Taking conditional expectations in~\eqref{eq:big-sum} and using~\eqref{eq:mu-t}, we obtain
\begin{align*}
 R_{n,t} = \Exp ( Y_{n,t} \mid \cF_n ) - \Exp Y_{n,t} & = \sum_{i=1}^{n+1} \mu ( t / L_{n,i} ) + n - \mu (t) \\
& = \sum_{i=1}^{n+1} \left( \frac{2L_{n,i}}{t} - 1 \right) \1 { L_{n,i} > t } + n +1 - \frac{2}{t} \\
& =  \sum_{i=1}^{n+1}  \left( 1 - \frac{2L_{n,i}}{t}   \right)\1 { L_{n,i} \leq t } ,\end{align*}
using $\sum_{i=1}^{n+1} L_{n,i} = 1$ and $t \in (0,1)$. 
Thus for $g(u) = 1-2u$ we identify from~\eqref{eq:U-n-t-def} that $R_{n,t} =  \cK_n^g (0,t]$,
and since $| g (u) | \leq 1 $, we verify~\eqref{eq:Rnt-alternative} using~\eqref{eq:K-n-t-def}.
\end{proof}

\begin{remark}
\label{rems:R-CLT}
The bound $|R_{n,t}| \leq K_{n,t}$ from~\eqref{eq:Rnt-alternative}
shows that $| R_{n,t} | = O ( n^2 t)$
with high probability. This bound is poor, since~\eqref{eq:R-S}
and Lemma~\ref{lem:S-bound}
say $\Var R_{n,t} = \Exp S_{n,t} = O ( n^2 t)$,
and $\Exp R_{n,t} =0$, so one expects $|R_{n,t}|$ to be around $O ( n t^{1/2} )$. Indeed, if   $\theta \in (0,1)$,
 the fluctuation results of Pyke \& van~Zwet  
 (Theorem~6.2 of~\cite{pvz}) show that $n^{-1/2} R_{n,\theta/n}$
 has a Gaussian limit. 
 However, when $n t \to 0$ this result says only that $n^{-1/2} R_{n,t} \to 0$ in probability. 
Proposition~\ref{prop:R-S-cross-moments} below includes moments asymptotics $\Exp ( R_{n,t}^p )$ that address these points, giving finer control on the asymptotics of $R_{n,t}$ for a broader range of~$t$.
\end{remark}

\section{Conditional means, variances, and their moments}
\label{sec:RS-moments}

The aim of this section is to establish the following asymptotics
on the mixed moments of $R_{n,t}$ and $S_{n,t}$ defined at~\eqref{eq:R-def} and~\eqref{eq:S-def} respectively.
The result is in two parts, depending on the parity of the exponent of~$R_{n,t}$; recall that $\Exp R_{n,t} = 0$.

\begin{proposition}
    \label{prop:R-S-cross-moments}
Suppose that $\nu \in (1, \frac{3}{2})$. Then the following hold:
\begin{thmenumi}[label=(\roman*)]
\item 
\label{prop:R-S-cross-moments-i}
There exist constants $C < \infty$ and $\delta >0$ such that, for all $n \in \N$,
all $t \in (n^{-\nu}, \frac{1}{4(n+1)} )$, and all $p \in 2\ZP$ and $q \in \ZP$
with $1 \leq p+q \leq \delta \log n$, 
\begin{equation}
\label{eq:R-S-p-even-bound}
     \left|    \Exp \left( R_{n,t}^{p} S_{n,t}^q \right) - 
    \frac{p!}{2^{\frac{p}{2}} (\frac{p}{2})!} \left( \frac{n^2 t}{6} \right)^{q+\frac{p}{2}}  \right|   \leq  C   \left( \tfrac{p}{2}+1 \right)! \cdot ( 40 n^2 t)^{q +\frac{p}{2}} \cdot (n^2 t)^{-1/2}.
\end{equation}
    \item 
\label{prop:R-S-cross-moments-ii}
There exist constants $C < \infty$ and $\delta >0$ such that, for all $n \in \N$,
all $t \in ( n^{-\nu}, \frac{1}{4(n+1)} )$, and all $p-1 \in 2\ZP$ and $q \in \ZP$
with $1 \leq p+q \leq \delta \log n$, 
\begin{equation}
\label{eq:R-S-p-odd-bound}
        \left|   \Exp \left( R_{n,t}^{p} S_{n,t}^q \right) -  
        \frac{q \cdot (p+1)!}{2^{\frac{p+3}{2}} \cdot (\frac{p+1}{2})!}  \left( \frac{n^2 t}{6} \right)^{q+\frac{p-1}{2}} \right| \leq
         C   \left( \tfrac{p+1}{2} \right)! \cdot ( 40 n^2 t)^{q +\frac{p-1}{2}} \cdot (n^2 t)^{-1/2}.
    \end{equation}
\end{thmenumi}
\end{proposition}
 
We give the proof of Proposition \ref{prop:R-S-cross-moments} later in this section. Lemma~\ref{lem:small-gaps-U}, on statistics of small gaps,
combined with 
Lemmas~\ref{lem:R-alternative} and~\ref{lem:S-bound} which represent, respectively, $R_{n,t}$ and $S_{n,t}$ in terms of small-gap functionals, 
enables us to represent $R_{n,t}$ and $S_{n,t}$ in the form
\begin{equation}
\label{eq:R-S-rep}
R_{n,t} = \sum_{i=1}^{K_n (0,t]} u_i, \text{ and }
S_{n,t} = s_0 \sum_{i=1}^{K_{n} (0,t]} \frac{1+u_i}{2} - \sum_{j=1}^{K_n (t,2t]} w ( 1 / v_j ), \end{equation}
where $u_i \sim \unif{-1}{1}$, $v_j \sim \unif{1}{2}$, and,
\emph{conditional} on the $K_n (0,t]$ and $K_n (t,2t]$, the random variables 
$u_1, u_2, \ldots$ and $v_1, v_2, \ldots$ are all mutually independent.
Write
    \begin{equation}
        \label{eq:W-def}
 W_{n,t} := - \sum_{j=1}^{K_n (t,2t]} w (1/v_j) .
     \end{equation}
Then~\eqref{eq:R-S-rep} is equivalent to
\begin{equation}
\label{eq:S-expression}
R_{n,t} = \sum_{i=1}^{K_{n,t}} u_i, \text{ and }
S_{n,t} = \frac{s_0}{2} K_{n,t}  + \frac{R_{n,t}}{2}   + W_{n,t} . \end{equation}
The expressions for the moments of $K_{n,t}$ from
Lemma~\ref{lem:small-gaps-K}, with the representation~\eqref{eq:S-expression} and the associated conditional
independence structure, made explicit in Lemma~\ref{lem:K-R-W} below, is our
starting point for the proof of 
Proposition~\ref{prop:R-S-cross-moments}.

\begin{remark}
    \label{rem:R-S-algebra}
With $w$ as defined at~\eqref{eq:w-def}, and $v_1 \sim \unif{1}{2}$, some calculus shows that the expectation of the each summand appearing in~\eqref{eq:W-def} is
\begin{equation}
\label{eq:gamma-def}
\Exp w (1/v_1) = \int_1^2 w (1/v) \ud u = - \gamma, \text{ where } \gamma := \frac{17}{6} - 4 \log 2 \approx 0.060745, \end{equation}
which, by comparison with the formula for $s_0$ from Proposition~\ref{prop:variance}, shows that 
\begin{equation}
    \label{eq:v0-gamma}
s_0 =  \frac{2}{3} -  2 \gamma.
\end{equation}
Thus from the representation~\eqref{eq:S-expression} we confirm that $\Exp R_{n,t} = 0$ (as is clear from~\eqref{eq:R-def}) and 
\[ \Exp S_{n,t} = \left( \frac{s_0}{2} + \frac{17}{6} - 4 \log 2 \right) \Exp K_{n,t}
= \frac{\Exp K_{n,t}}{3}  . \]
Moreover, it also follows from~\eqref{eq:S-expression} that 
\begin{align*} \Var R_{n,t} & = \Exp  \Var ( R_{n,t} \mid K_{n,t} ) + \Var \Exp ( R_{n,t} \mid K_{n,t} ) = \Exp ( K_{n,t} \cdot \Var u_1 )   = \frac{\Exp K_{n,t}}{3}  ,
\end{align*}
so the  relation~\eqref{eq:v0-gamma} recovers~$\Exp S_{n,t} = \Var R_{n,t}$, as at~\eqref{eq:R-S}.
\end{remark}

The next result summarizes the structure of the components of $S_{n,t}$ expressed in~\eqref{eq:S-expression}.

\begin{lemma}
    \label{lem:K-R-W}
    Let $n \in \N$ and $t \in (0,\frac{1}{2(n+1)} )$. 
Conditional on $K_{n,2t}  = k \in \ZP$,
the random variables $K_{n,t}$, $R_{n,t}$, $W_{n,t}$ have the representation
\[
\left( K_{n,t} , R_{n,t}, W_{n,t} \right) \eqd
\left( K, \sum_{i=1}^K u_i, - \sum_{j=1}^{k-K} w (1/v_j) \right) ,
\]
where $K \sim \Bin {k}{1/2}$, the $u_i \sim \unif{-1}{1}$, and $v_j \sim \unif{1}{2}$ are all independent.
\end{lemma}
\begin{proof}
The representation for $R_{n,t}$ from~\eqref{eq:S-expression}
(coming from Lemma~\ref{lem:R-alternative} and Lemma~\ref{lem:small-gaps-U}) together with the definition of $W_{n,t}$ from~\eqref{eq:W-def} gives
\[ 
\left( K_{n,t} , R_{n,t}, W_{n,t} \right) \eqd
\left( K_n (0,t], \sum_{i=1}^{K_n (0,t]} u_i, - \sum_{j=1}^{K_n (t,2t]} w (1/v_j) \right) ,
\]
where, given $K_n (0,t]$ and $K_n (t,2t]$, the $u_i, v_j$ are all
mutually independent with $u_i \sim \unif{-1}{1}$
and $v_j \sim \unif{1}{2}$.
By~\eqref{eq:K-n-t-def},
$K_n (t,2t] = K_n (0,2t] - K_n (0,t]$ and,
with the notation from~\eqref{eq:U-n-t-def}, 
$K_n (0,t] = \cK_n^g (0,2t]$
with $g (u) := \1 { u \leq 1/2 }$. Hence 
Lemma~\ref{lem:small-gaps-U} shows that, given $K_n (0,2t] = k$,
$K_{n} (0,t] \sim \Bin{k}{1/2}$.
\end{proof}

We will use Lemma~\ref{lem:K-R-W}
to obtain, in Lemma~\ref{lem:K-R-W-moments} below,
estimates for the mixed moments of
$K_{n,t}$, $R_{n,t}$, and $W_{n,t}$. In the proof,
we will make use of auxiliary results stated in Appendix \ref{sec:appendix} below, including estimates of moments of random sums, 
like those appearing in the triple representation
in Lemma~\ref{lem:K-R-W},
given in Lemma~\ref{lem:K-statistic-moments}.

\begin{lemma}
\label{lem:K-R-W-moments}
\begin{thmenumi}[label=(\roman*)]
\item
\label{lem:K-R-W-moments-i}
Suppose that
    $n \in \N$ and $t \in (0,\frac{1}{2(n+1)} )$. 
 Then, for $a, b, c \in \ZP$, 
        \begin{align}
\label{eq:K-R-W_odd}
      \Exp \bigl( K_{n,t}^a R_{n,t}^{2b+1} W_{n,t}^c \bigr) & = 0 .
      \end{align}
\item
There is a constant $C <\infty$ such that for all $n \in \N$, all $t \in (0,\frac{1}{2(n+1)} )$, and all $a, b, c \in \ZP$
with $(a+b+c)^2 \leq n^2 t$, 
\label{lem:K-R-W-moments-ii}
\begin{align}
         \label{eq:K-R-W_even_crude}
            \Exp \bigl| K_{n,t} ^a R_{n,t}^{2b} W_{n,t}^c \bigr| \leq C \frac{(2b)!}{b!}  (4 n^2 t)^{a+b+c} .
    \end{align}
\item
\label{lem:K-R-W-moments-iii}
Let $\nu \in (1,\frac{3}{2})$.  There are  constants $\delta>0$ and  $C < \infty$ such that, 
for all $n \in \N$, for all $t \in ( n^{-\nu}, \frac{1}{4(n+1)} )$, and for all~$a, b, c \in \ZP$ with $1 \leq a+b+c \leq \delta \log n$, 
      \begin{align}
         \label{eq:K-R-W_even}
           \left|  \Exp \bigl( K_{n,t} ^a R_{n,t}^{2b} W_{n,t}^c \bigr) - \frac{(2b)! \cdot \gamma^c \cdot  (n^2 t)^{a+b+c}}{2^{a+b+c} \cdot 6^b \cdot b!} \right| \leq C  (b+1)!  (16 n^2 t)^{a+b+c-(1/2)}. 
    \end{align}
    \end{thmenumi}
\end{lemma}

In the following proof, and frequently later on, we will need simple inequalities relating $(2n)!$ and $n!$ derived from
\begin{equation}
\label{eq:simple-binomial}
\binom{2n}{n} \leq 2^{2n} = \sum_{k=0}^{2n} \binom{2n}{k} \leq (2n+1) \binom{2n}{n} .
\end{equation}
\begin{proof}[Proof of Lemma~\ref{lem:K-R-W-moments}]
 Lemma~\ref{lem:K-R-W} shows that
$(K_{n,t}, R_{n,t}, W_{n,t})$ and 
$(K_{n,t}, -R_{n,t}, W_{n,t})$
have the same distribution, and this yields~\eqref{eq:K-R-W_odd}, and hence
proves part~\ref{lem:K-R-W-moments-i}.

For parts~\ref{lem:K-R-W-moments-ii}--\ref{lem:K-R-W-moments-iii}, 
denote, similarly to Lemma~\ref{lem:K-statistic-moments}, the moments  
\[ \mu_k (n) := \Exp \biggl( \Bigl( \sum_{i=1}^n u_i \Bigr)^k \biggr), \text{ and }
\mu'_k (n) := \Exp \biggl( \Bigl( \sum_{j=1}^n w_j \Bigr)^k \biggr),
\]
for independent sequences $u_i \sim \unif{-1}{1}$ and $w_j := -w (1/v_j)$, $v_j \sim \unif{1}{2}$. 
By Lemma~\ref{lem:K-R-W}, 
\begin{align*}
    \Exp \bigl| K_{n,t}^{a} R_{n,t}^{2b} W_{n,t}^c \bigr| & 
    = \Exp \left( \Exp \left( \bigl| K_{n,t}^{a} R_{n,t}^{2b} W_{n,t}^c \bigr| \mid K_{n,t}, K_{n,2t} \right) \right) \\
    & = \Exp \left( \Exp \left( K_{n,t}^a \Exp ( R_{n,t}^{2b} \mid K_{n,t} ) \Exp ( |W_{n,t}|^c \mid K_{n,t}, K_{n, 2t} ) \mid K_{n,2t} \right) \right) \\
    & = \Exp \left( \Exp \left( K_{n,t}^a \mu_{2b} ( K_{n,t} ) \Exp ( |W_{n,t}|^c \mid K_{n,t}, K_{n, 2t} ) \mid K_{n,2t} \right) \right).
\end{align*} 
Since $\sup_{0 \leq t \leq 1} | w(t) | = s_0$, 
 $| W_{n,t} | \leq s_0 K_{n,2t}$ by~\eqref{eq:W-def}, and $K_{n,t} \leq K_{n,2t}$ by~\eqref{eq:K-n-t-def}, so
\begin{align}
\label{eq:KRW-to-K}
    \Exp \bigl| K_{n,t}^{a} R_{n,t}^{2b} W_{n,t}^c \bigr| & 
     \leq s_0^c \Exp  \left( K_{n,2t}^{a+c} \mu_{2b} ( K_{n,t} )  \right)
     \leq \Exp  \left( K_{n,2t}^{a+c} \mu_{2b} ( K_{n,t} )  \right)
     ,
\end{align}
since $s_0 \in (0,1)$.
 By Lemma~\ref{lem:K-statistic-moments}\ref{lem:K-statistic-moments-i} applied with $\xi \sim \unif{-1}{1}$,
noting that $\Exp ( \xi^r ) = \frac{1}{1+r}$ for $r \in \ZP$, there is a constant $C < \infty$ such that, for all $b \in \ZP$ and all $n \in \N$,
\begin{align}
\label{eq:mu-2b}
    \left| \mu_{2b} ( n ) - \frac{(2b)!}{6^b b!} n^b \right| \leq C b \cdot  \frac{(2b)!}{2^b b!} n^{b-1} 
    \leq C 2^b b \cdot b! \cdot n^{b-1} ,
\end{align}
by~\eqref{eq:simple-binomial}. 
From~\eqref{eq:KRW-to-K}--\eqref{eq:mu-2b} there is $C<\infty$ such that, for all $n \in \N$ and all $a,b,c \in \ZP$,
\begin{align*}
    \Exp \bigl| K_{n,t}^{a} R_{n,t}^{2b} W_{n,t}^c \bigr| 
    \leq C \frac{(2b)!}{2^b b!} \Exp \left( K_{n,2t}^{a+b+c} + b K_{n,2t}^{a+b+c-1} \right).
\end{align*}
Now using Lemma~\ref{lem:small-gaps-K}\ref{lem:small-gaps-K-i}, we
verify part~\ref{lem:K-R-W-moments-ii}.

Finally, for part~\ref{lem:K-R-W-moments-iii}, suppose $a+b+c \geq 1$. From Lemma~\ref{lem:K-R-W}, with   $K \sim \Bin{k}{1/2}$,
\begin{align}
\label{eq:triple-mu}
\Exp ( K_{n,t}^a R_{n,t}^{2b} W_{n,t}^c \mid  K_n (0,2t]  = k)
& = \Exp \left( K^a \mu_{2b} (K ) \mu'_c ( k - K ) \right) . \end{align}
Here,  by Lemma~\ref{lem:K-statistic-moments}\ref{lem:K-statistic-moments-ii} applied with $\xi =  - w (1/v)$, $v \sim \unif{1}{2}$,
and the fact that $| w_j | \leq s_0 < 3/5$ and $\Exp \xi = \gamma \in (0,3/5)$, by~\eqref{eq:gamma-def},
for all $c \in \ZP$ and all $n \in \N$,
\begin{align*}
    \left| \mu'_{c} ( n ) - \gamma^c n^c \right| \leq 3 c^2 (3/5)^c n^{c-1}.
\end{align*}
Combining this bound with~\eqref{eq:mu-2b}, using
$p q - r s = r (q-s) + s(p-r) + (p-r)(q-s)$, we obtain, for some $C < \infty$ and all $b, c \in \ZP$ and all $n, m \in \N$, 
\begin{align}
\label{eq:two-mu-bound}
    \left| \mu_{2b} ( n) \mu'_c (m) - \frac{(2b)!}{6^b b!} \gamma^c n^b m^c \right|
    & \leq 
    C (4/5)^c \cdot 2^b b! \cdot \left(    3^{-b}  n^b m^{c-1} \2 { c \geq 1}
    +   b  m^c    n^{b-1}   \2 { b \geq 1} \right).
\end{align}
Using the bound~\eqref{eq:two-mu-bound} in~\eqref{eq:triple-mu}, we get, for $K \sim \Bin{k}{1/2}$,
\begin{align*}
&{} \left|    \Exp ( K_{n,t}^a R_{n,t}^{2b} W_{n,t}^c \mid  K_n (0,2t]  = k)
    - \frac{(2b)!}{6^b b!} \gamma^c \Exp ( K^{a+b} (k-K)^c  ) \right|\\
    & {} \qquad {} \leq C   b! \cdot \Exp ( K^{a+b} (k-K)^{c-1}   )  \2 { c \geq 1 }    + C 2^b b   \cdot b! \cdot \Exp ( K^{a+b-1} (k-K)^{c}  ) \2 {b \geq 1} \\ 
    & {} \qquad {} \leq C2^b ( b+1)! \max (a+b, c) k^{a+b+c-1},
\end{align*}
by an application of Lemma~\ref{lem:binomial-product}. Another application of Lemma~\ref{lem:binomial-product} shows that 
    \[ 
\left|     \Exp ( K^{a+b} (k-K)^c   ) - (k/2)^{a+b+c} \right|
\leq \max (a+b, c) k^{a+b+c-(1/2)}.
    \]
    It follows that
\begin{align*}
&{}\left|    \Exp ( K_{n,t}^a R_{n,t}^{2b} W_{n,t}^c  )
    -  \frac{(2b)! \gamma^c}{2^{a+b+c} \cdot 6^b \cdot b!} \Exp (  K_{n,2t}^{a+b+c} ) \right| \\
    & {} \qquad {} 
    \leq C 2^b (b+1)!  \max (a+b, c) \left\{
    \Exp (  K_{n,2t}^{a+b+c-(1/2)} ) + \Exp (  K_{n,2t} ^{a+b+c-1} ) \right\} \\
    & {} \qquad {} 
    \leq C (b+1)!  (16 n^2 t)^{a+b+c-(1/2)} ,
    \end{align*}
    by Lemma~\ref{lem:small-gaps-K}\ref{lem:small-gaps-K-i},
    provided that $(a+b+c)^2 \leq n^2 t$. 
    For $\nu \in (1,\frac{3}{2})$, take $\delta >0$ as in Lemma~\ref{lem:small-gaps-K}\ref{lem:small-gaps-K-ii}. 
Assuming that $t \geq n^{-\nu}$ and $a+b+c \leq \delta \log n$, we have that, indeed, $(a+b+c)^2 \leq n^2 t$ for all $n$ large enough. Moreover, applying  Lemma~\ref{lem:small-gaps-K}\ref{lem:small-gaps-K-ii} to estimate $\Exp (  K_{n,2t}^{a+b+c} )$, we obtain for all $t \in [n^{-\nu}, \frac{1}{4(n+1)}]$,
\begin{align*}
    \left|  \Exp ( K_{n,t}^a R_{n,t}^{2b} W_{n,t}^c  )
    -  \frac{(2b)! \gamma^c (n^2 t)^{a+b+c}}{2^{a+b+c} \cdot 6^b \cdot b!}  \right| \leq E^{(1)}_{n,t} + E^{(2)}_{n,t} ,
\end{align*}
where $E^{(1)}_{n,t} =  C (b+1)!  (16 n^2 t)^{a+b+c-(1/2)}$ and
\begin{align*} E^{(2)}_{n,t} & \leq  \frac{C (2b)! \cdot \gamma^c \cdot 4^{a+b+c} \cdot (a+b+c)}{2^{a+b+c} \cdot 6^b \cdot b!} ( n^2 t)^{a+b+c} n^{-1/2}  \\
& \leq C (b+1)! (4 n^2 t)^{a+b+c} n^{-1/2} ,
\end{align*}
using that $\gamma < 1$ and $(2b)!/b! \leq 4^b b!$, by~\eqref{eq:simple-binomial}. 
Hence,  for all $t \in [n^{-\nu}, \frac{1}{4(n+1)}]$,
\[ E^{(2)}_{n,t} \leq C (b+1)! (4 n^2 t )^{a+b+c - (1/2)} (n t)^{1/2} ,\]
and since $n t \leq 1$ we obtain~\eqref{eq:K-R-W_even}. This proves part~\ref{lem:K-R-W-moments-iii}. 
\end{proof}

\begin{proof}[Proof of Proposition~\ref{prop:R-S-cross-moments}]
Let $\nu \in (1,\frac{3}{2})$ and $\delta >0$ as in Lemma~\ref{lem:K-R-W-moments}\ref{lem:K-R-W-moments-iii}. Suppose $p, q \in \ZP$ with $1 \leq p + q \leq \delta \log n$.
    We can use~\eqref{eq:S-expression} and a trinomial expansion to write
    \begin{equation}
        \label{eq:S-trinomial}
     S_{n,t}^q = \sum_{\substack{a+b+c =q}} \frac{q!}{a! b! c!} \frac{s^a_0}{2^{a+b}} K_{n,t}^a R_{n,t}^b W_{n,t}^c ,
    \end{equation}
        where $W_{n,t}$ is defined at~\eqref{eq:W-def}, and the sum is over
        $a,b,c \in \ZP$ whose sum is equal to~$q$.
   Taking expectations in~\eqref{eq:S-trinomial} and applying~\eqref{eq:K-R-W_odd},
   we obtain
\begin{align}
\label{eq:R-S-to_K-R-W}
     \Exp \left( R_{n,t}^{p} S_{n,t}^q \right) & = 
     \sum_{\substack{a+b+c =q \\ p+b \in 2 \Z}} \frac{q!}{a! b! c!} \frac{ s^a_0}{2^{a+b}} \Exp \bigl( K_{n,t}^a R_{n,t}^{p+b} W_{n,t}^c  \bigr) .\end{align}
Provided $t \in ( n^{-\nu}, \frac{1}{2(n+1)})$, we have $n^2 t \geq n^{2-\nu}$ and so $a+\frac{p+b}{2}+c = q + \frac{p}{2} - \frac{b}{2} \leq p+q$ where $(p+q)^2 \leq \delta^2 (\log n)^2 \leq n^2 t$ for all but finitely many $n\in\N$. Hence the hypotheses of parts~\ref{lem:K-R-W-moments-ii} and~\ref{lem:K-R-W-moments-iii} of Lemma~\ref{lem:K-R-W-moments} are both satisfied when considering $\Exp ( K_{n,t}^a R_{n,t}^{p+b} W_{n,t}^c )$. 
     
     First suppose that $p \in 2\ZP$; note that $q+\frac{p}{2} \geq 1$ in this case.
     In the sum in~\eqref{eq:R-S-to_K-R-W}, we show that the terms with $b=0$ are dominant. To this end, consider
\begin{align*}
& {} \left|    \Exp \bigl( R_{n,t}^{p} S_{n,t}^q \bigr) - 
    \sum_{\substack{a+c =q}} \frac{q!}{a! c!} \frac{ s^a_0}{2^{a}} \Exp \bigl( K_{n,t}^a R_{n,t}^{p} W_{n,t}^c  \bigr) \right|\\
  & {} \qquad {} \leq \sum_{b \in 2 \N, \, a+b+c=q}
  \frac{q!}{a! b! c!} \frac{ s^a_0}{2^{a+b}} \Exp \bigl| K_{n,t}^a R_{n,t}^{p+b} W_{n,t}^c  \bigr| \\
  & {} \qquad {} \leq C ( 4n^2 t)^{q +\frac{p}{2}} \sum_{b \in 2 \N, \, b \leq q} \binom{q}{b}
  \frac{(p+b)!}{2^b \left( \frac{p+b}{2} \right)!}  (n^2 t)^{-\frac{b}{2}} 
  \sum_{a=0}^{q-b} \binom{q-b}{a} 
 \frac{ s^a_0}{2^{a}} ,
\end{align*}
using the upper bound from Lemma~\ref{lem:K-R-W-moments}\ref{lem:K-R-W-moments-ii}. Since $s_0/2 < 1$,
using~\eqref{eq:simple-binomial},
\begin{align*}
& {} \left|    \Exp \left( R_{n,t}^{p} S_{n,t}^q \right) - 
    \sum_{\substack{a+c =q}} \frac{q!}{a! c!} \frac{ s^a_0}{2^{a}} \Exp \left( K_{n,t}^a R_{n,t}^{p} W_{n,t}^c  \right) \right|\\
  & {} \qquad {} \leq  C   ( 8 n^2 t)^{q +\frac{p}{2}} \sum_{b \in 2 \N, \, b \leq q}    \left( \tfrac{p+b}{2} \right)!
  (n^2 t)^{-\frac{b}{2}} \\
   & {} \qquad {} \leq  C    \left( \tfrac{p}{2} \right)! \cdot ( 8 n^2 t)^{q +\frac{p}{2}} \sum_{b \in 2 \N, \, b \leq q}    \left( \tfrac{p+q}{2} \right)^b
  (n^2 t)^{-\frac{b}{2}} 
,
\end{align*}
since $(x+y)! \leq (x+y)^y x!$ for $x,y \in \ZP$. Now
\begin{align*}
\sum_{b \in 2 \N, \, b \leq q}    \left( \tfrac{p+q}{2} \right)^b
  (n^2 t)^{-\frac{b}{2}}  
& \leq \left( \tfrac{p+q}{2} \right)^2 (n^2 t)^{-1} 
\sum_{\ell = 0}^\infty
  \left(\frac{(p+q)^2}{4n^2 t}\right)^{\ell}   \leq (p+q)^2 (n^2 t)^{-1} ,
\end{align*}
using the fact that  $(p+q)^2 \leq  n^2 t$. Thus, for some $C < \infty$ and all $p, q$ with $p+q \leq \delta \log n$,
\begin{align}
\label{eq:R-S-even}
   \left|    \Exp \left( R_{n,t}^{p} S_{n,t}^q \right) - 
    \sum_{\substack{a+c =q}} \frac{q!}{a! c!} \frac{ s^a_0}{2^{a}} \Exp \left( K_{n,t}^a R_{n,t}^{p} W_{n,t}^c  \right) \right|   \leq  C   \left( \tfrac{p}{2} \right)! \cdot ( 9 n^2 t)^{q +\frac{p}{2}-1} .
\end{align}
     On the other hand, from Lemma~\ref{lem:K-R-W-moments}\ref{lem:K-R-W-moments-iii} it follows that 
\begin{align*}
& \left|  \sum_{\substack{a+c =q}} \frac{q!}{a! c!} \frac{ s^a_0}{2^{a}} \Exp \left( K_{n,t}^a R_{n,t}^{p} W_{n,t}^c  \right) 
-  \sum_{\substack{a+c =q}} \frac{q!}{a! c!} \frac{ p! \cdot s^a_0 \cdot \gamma^c}{2^{2a+c} \cdot 12^{\frac{p}{2}} \cdot \left( \tfrac{p}{2} \right)!} ( n^2 t)^{q+\frac{p}{2}} \right| \\
& {} \qquad {} \leq C \sum_{\substack{a+c =q}}   \frac{q!}{a! c!} \frac{ s^a_0}{2^{a}} \left( \tfrac{p}{2}+1 \right)! ( 16 n^2 t)^{q+\frac{p}{2}-\frac{1}{2}}   \leq C 32^{q+\frac{p}{2}}   \left( \tfrac{p}{2}+1 \right)! ( n^2 t)^{q+\frac{p}{2}-\frac{1}{2}} .
\end{align*}
Furthermore,
\begin{align*}
    \sum_{\substack{a+c =q}} \frac{q!}{a! c!} \frac{ p! \cdot s^a_0 \cdot \gamma^c}{2^{2a+c} \cdot 12^{\frac{p}{2}} \cdot \left( \tfrac{p}{2} \right)!} ( n^2 t)^{q+\frac{p}{2}}
    & =   \frac{p! }{2^q \cdot 12^{\frac{p}{2}} \cdot (\frac{p}{2} )!} (n^2 t)^{q+\frac{p}{2}} 
       \sum_{a = 0}^q  \binom{q}{a} \frac{s^a_0 \gamma^{q-a}}{2^{a}} 
  \\
     & = \frac{p! }{2^q \cdot 12^{\frac{p}{2}} \cdot (\frac{p}{2} )!} (n^2 t)^{q+\frac{p}{2}} 
       \left( \frac{s_0}{2} + \gamma \right)^q 
  \\
     & = \frac{p!}{2^{\frac{p}{2}} (\frac{p}{2})!} \left( \frac{n^2 t}{6} \right)^{q+\frac{p}{2}} 
   ,
     \end{align*}
     using the relation~\eqref{eq:v0-gamma} between $s_0$ and $\gamma$.
Combined with~\eqref{eq:R-S-even}, we verify~\eqref{eq:R-S-p-even-bound}.

Finally, suppose that $p-1 \in 2\ZP$. 
Then
     if $q=0$ we have $\Exp ( R_{n,t}^p ) =0$ and~\eqref{eq:R-S-p-odd-bound} is trivial, 
     so suppose that $q \in \N$; note that $q+\frac{p}{2} \geq \frac{3}{2}$ in this case. 
          In the sum in~\eqref{eq:R-S-to_K-R-W}, we show that now the terms with $b=1$ are dominant. Now, by~\eqref{eq:R-S-to_K-R-W}, 
          \begin{align*}
& {} \left|    \Exp \bigl( R_{n,t}^{p} S_{n,t}^q \bigr) - 
    \sum_{\substack{a+c =q-1}} \frac{q!}{a! c!} \frac{ s^a_0}{2^{a+1}} \Exp \bigl( K_{n,t}^a R_{n,t}^{p+1} W_{n,t}^c  \bigr) \right|\\
  & {} \qquad {} \leq \sum_{b-1 \in 2 \N, \, a+b+c=q}
  \frac{q!}{a! b! c!} \frac{ s^a_0}{2^{a+b}} \Exp \bigl| K_{n,t}^a R_{n,t}^{p+b} W_{n,t}^c  \bigr| \\
  & {} \qquad {} \leq C  ( 4 n^2 t)^{q +\frac{p}{2}} \sum_{b - 1 \in 2 \N, \, b \leq q} \binom{q}{b}
  \frac{(p+b)!}{2^b \left( \frac{p+b}{2} \right)!}  (n^2 t)^{-\frac{b}{2}} 
  \sum_{a=0}^{q-b} \binom{q-b}{a} 
 \frac{ s^a_0}{2^{a}} ,
\end{align*}
using the upper bound from Lemma~\ref{lem:K-R-W-moments}\ref{lem:K-R-W-moments-ii}. Now following a similar argument to that leading to~\eqref{eq:R-S-even} we obtain, for some $C < \infty$ and all $p, q$ with $p+q \leq \delta \log n$,
\begin{align}
\label{eq:R-S-odd}
   \left|    \Exp \left( R_{n,t}^{p} S_{n,t}^q \right) - 
    \sum_{\substack{a+c =q-1}} \frac{q!}{a! c!} \frac{ s^a_0}{2^{a+1}} \Exp \left( K_{n,t}^a R_{n,t}^{p+1} W_{n,t}^c  \right) \right|   \leq  C   \left( \tfrac{p}{2} \right)! \cdot ( 9 n^2 t)^{q +\frac{p}{2}-\frac{3}{2}} .
\end{align}
     On the other hand, from Lemma~\ref{lem:K-R-W-moments}\ref{lem:K-R-W-moments-iii} it follows that 
\begin{align*}
& \left|  \sum_{\substack{a+c =q-1}} \frac{q!}{a! c!} \frac{ s^a_0}{2^{a+1}} \Exp \left( K_{n,t}^a R_{n,t}^{p+1} W_{n,t}^c  \right) 
-  \sum_{\substack{a+c =q-1}} \frac{q!}{a! c!} \frac{ (p+1)! \cdot s^a_0 \cdot \gamma^c}{2^{2a+c+1} \cdot 12^{\frac{p+1}{2}} \cdot \left( \tfrac{p+1}{2} \right)!} ( n^2 t)^{q+\frac{p-1}{2}} \right| \\
& {} \qquad {} \leq C \sum_{\substack{a+c =q-1}}   \frac{q!}{a! c!} \frac{ s^a_0}{2^{a}}   \left( \tfrac{p+1}{2}+1 \right)! ( 16 n^2 t)^{q+\frac{p}{2}-1}   \leq C 32^{q+\frac{p}{2}}   \left( \tfrac{p+1}{2}+1 \right)! ( n^2 t)^{q+\frac{p}{2}-1} .
\end{align*}
Furthermore,
\begin{align*}
    \sum_{\substack{a+c =q-1}} \frac{q!}{a! c!} \frac{ (p+1)! \cdot s^a_0 \cdot \gamma^c}{2^{2a+c+1} \cdot 12^{\frac{p+1}{2}} \cdot \left( \tfrac{p+1}{2} \right)!} 
    & =   \frac{q (p+1)! }{2^q \cdot 12^{\frac{p+1}{2}} \cdot (\frac{p+1}{2} )!} 
       \sum_{a = 0}^{q-1}  \binom{q-1}{a} \frac{s^a_0 \gamma^{q-1-a}}{2^{a}} 
  \\
     & = \frac{q (p+1)! }{2^q \cdot 12^{\frac{p+1}{2}} \cdot (\frac{p+1}{2} )!}  
       \left( \frac{s_0}{2} + \gamma \right)^{q-1} 
  \\
     & = \frac{3 q (p+1)!}{6^q \cdot 12^{\frac{p+1}{2}} (\frac{p+1}{2})!}  
   ,
     \end{align*}
     using~\eqref{eq:v0-gamma}.
Combined with~\eqref{eq:R-S-odd}, we verify~\eqref{eq:R-S-p-odd-bound}.  
\end{proof}

\section{Completing the proofs of the main  theorems}
\label{sec:hermite}

In this section we combine the ingredients developed so far with an expansion in Hermite polynomials to prove our quantitative CLTs for $N_t$ (Theorem~\ref{thm:n-b-e}) and $M_n$ (Theorem~\ref{thm:clt}). A key intermediate result is Proposition~\ref{prop:clt-with-error} below.
Recall the definitions of $v(t)$, $R_{n,t}$ and $S_{n,t}$ from~\eqref{eq:var-t}, \eqref{eq:R-def} and~\eqref{eq:S-def}, respectively, and 
for $\ell \in \ZP$ define
\begin{equation}
\label{eq:E-def}    
\cE^\ell _{n,t}  := v(t)^{-\frac{\ell+1}{2}} \sum_{k=0}^{\lfloor(\ell+1)/2\rfloor}\frac{(-1)^{k+1}}{2^k k! (\ell-2k+1)!}
 R_{n,t}^{\ell-2k+1} S_{n,t}^k. 
\end{equation}
The following result reduces the proof of Theorem~\ref{thm:n-b-e} to controlling (sums of) the quantities $\cE^\ell_{n,t}$ from~\eqref{eq:E-def}. Let $\Gamma$ denote the Euler gamma function, so $\Gamma (x+1) = x!$, $x \in \ZP$.

\begin{proposition}
    \label{prop:clt-with-error}
Let $\nu \in (1,\frac{3}{2})$, and let $\delta >0$ be as in Proposition~\ref{prop:R-S-cross-moments}. Then there exist  constants $c_0 >0$ and $C < \infty$ such that, for all $n \in \N$ and $t >0$ such that $n^{1-\nu} \leq nt \leq c_0$,  
\begin{align}
\label{eq:clt-with-error}
  \sup_{x \in \R} \left| \Pr \left( \frac{N_t -\Exp N_t}{\sqrt { \Var N_t}} \leq x  \right) - \Phi (x) \right| & \leq 
  \frac{C}{n^2 t^{3/2}} + 
\sum_{\ell=0}^{2 \lfloor (\delta/3) \log n \rfloor + 1} 2^{\ell/2} \Gamma \bigl( \tfrac{\ell+1}{2} \bigr)  \left| \Exp \cE^\ell_{n,t}  \right| .
  \end{align}
\end{proposition}

The main additional element to Proposition~\ref{prop:clt-with-error} is a sort of Hermite--Edgeworth expansion. Denote by $H_n$ the Hermite polynomial of degree~$n \in \ZP$, which satisfies
\begin{equation}\label{eq:hermite}   \phi^{(n)} (x) := \frac{\ud^n}{\ud x^n} \phi (x) = (-1)^n H_n (x) \phi (x), \text{ for } x \in \R ,\end{equation}
where $\phi (x) := ( 2\pi)^{-1/2} \re^{-x^2/2}$
is the standard Gaussian density: see e.g.~\cite[\S20.2]{schilling}.
We will need the following inequality from~\cite[p.~78]{lukacs}:
\begin{equation}
    \label{eq:cramer}
    \sup_{x \in \R} \left| \phi (x) H_n (x) \right| \leq 2^{n/2} \Gamma \bigl( \tfrac{n+1}{2} \bigr) , \text{ for every } n \in \N;
\end{equation}
note that $|H_{2n} (0)| = 2^{-n} (2n)! / n!$ shows this bound is not far from optimal.

\begin{lemma}\label{lem:new-edgeworth}
Let $m \in \N$. Then, for
all $z \in \R$ and all $\alpha \in (-\infty,1/2]$,
\begin{align*}
& \sup_{x \in \R} \left| \Phi \left( \frac{x+z}{\sqrt{1-\alpha}} \right) - \Phi(x) - \phi(x)\sum_{\ell=0}^{2m-1} \sum_{k=0}^{\lfloor (\ell+1)/2 \rfloor} \frac{(-1)^{\ell-k} z^{\ell+1-2k}}{k! (\ell+1-2k)!}  \left(\frac{\alpha}{2}\right)^k H_{\ell} (x) \right| \\
& {} \qquad\hspace{8cm} {} \leq 2^{2m} \left( |\alpha|^{m+1} + \frac{\Gamma ( \frac{m+1}{2} )}{(m+1)!}  |z|^{m+1} \right) .
\end{align*}
\end{lemma}
\begin{proof}
For $z\in\R$ and $\alpha \in (-\infty,1/2]$, let $Y\sim\cN(-z,1-\alpha)$; then
\[ \Pr ( Y \leq x ) = \Phi \left( \frac{x+z}{\sqrt{1-\alpha}} \right) , \text{ for all } x \in \R. \]
Let $\Re \xi$ 
denote the real part
of  $\xi \in \mathbb{C}$.
The random variable $Y$ has characteristic function 
\begin{align*}
\varphi_Y(t) = \re^{-izt-(1-\alpha)(t^2/2)}
 = \re^{-t^2/2} \exp (  g_{z,\alpha} (t) ),
 \text{ where } g_{z,\alpha} (t) := \frac{\alpha t^2}{2}-izt. \end{align*}
Using the Taylor series for the exponential function with complex argument
\[
\left| \re^{\xi} - \sum_{j=0}^m \frac{\xi^j}{j!} \right| \leq \frac{|\xi|^{m+1}}{(m+1)!} \re^{\Re \xi}, \text{ for } \xi \in \mathbb{C} \text{ and } m \in \ZP, 
\]
we obtain, for $t \in \R$,
\begin{equation}
\label{eq:taylor}
   \left| \varphi_Y(t) - \re^{-t^2/2}   \sum_{j=0}^m \frac{g_{z,\alpha} (t)^j}{j!} \right|
   \leq \frac{|\alpha t^2|^{m+1} + 2^{m+1} |z t |^{m+1} }{(m+1)!}  \re^{-(1-\alpha)t^2/2} .
\end{equation}
A standard inversion formula (see e.g.~Theorem~3.2.1 of~\cite[p.\ 31]{lukacs}) gives
\begin{align*}
    \Pr ( Y \leq x ) & = \frac{1}{2} + \frac{i}{2\pi} \int_{-\infty}^\infty  \bigl( \re^{-itx} \varphi_Y (t) \bigr) \frac{\ud t}{t} , \text{ for every } x \in \R. 
\end{align*}
Using the same formula for the standard Gaussian characteristic function shows that
\begin{align*}
   \frac{1}{2} +  \frac{i}{2\pi}  \sum_{j=0}^m \frac{1}{j!}  \int_{-\infty}^\infty    \re^{-itx} \re^{-t^2/2} g_{z,\alpha} (t)^j   \frac{\ud t}{t}
    & = \Phi (x ) + \frac{i}{2\pi}  \sum_{j=1}^m \frac{1}{j!}  \int_{-\infty}^\infty  \re^{-itx} \re^{-t^2/2} g_{z,\alpha} (t)^j   \frac{\ud t}{t}.
\end{align*}
Hence, by~\eqref{eq:taylor},
\begin{align*}
 & \sup_{x \in \R} \left|  \Pr ( Y \leq x ) - \Phi (x) - \frac{i}{2\pi}  \sum_{j=1}^m \frac{1}{j!} \int_{-\infty}^\infty   \re^{-itx} \re^{-t^2/2} g_{z,\alpha} (t)^j  \frac{\ud t}{t} \right| \\
 & {} \qquad{}  \leq \frac{1}{\pi (m+1)!} \int_{0}^\infty  \bigl( | \alpha|^{m+1} t^{2m+1} + 2^{m+1} |z|^{m+1} t^m \bigr)  \re^{-(1-\alpha)t^2/2} \ud t \\
 & {} \qquad{} =  \frac{(1-\alpha)^{-(m+1)/2}}{\pi (m+1)!} \left( 2^{m}  \alpha^{m+1 } (1-\alpha)^{-(m+1)/2} \Gamma ( m+1) 
 + |z|^{m+1} 2^{(3m+1)/2}  \Gamma ( \tfrac{m+1}{2} ) 
 \right) \\
 & {} \qquad{} \leq 2^{2m} \left(  {|\alpha|^{m+1}}  +  \frac{\Gamma ( \frac{m+1}{2} )}{(m+1)!} {|z|^{m+1}}  \right) ,
\end{align*}
provided $\alpha \in (-\infty,1/2]$.
Now, using the binomial theorem to expand $g_{z,\alpha}(t)^j$, 
\begin{align*}
     \int_{-\infty}^\infty     \re^{-itx} \re^{-t^2/2} g_{z,\alpha} (t)^j  \frac{\ud t}{t} 
     = \sum_{k=0}^j \binom{j}{k} z^{j-k}    \left( \frac{\alpha}{2} \right)^k 
      \int_{-\infty}^\infty  \re^{-itx} \re^{-t^2/2} (-i)^{j-k} t^{j+k}  \frac{\ud t}{t}.
\end{align*}
For $n \in \ZP$ we have the equality
\begin{align*}
\frac{1}{2\pi}\int_{-\infty}^\infty \re^{-itx}\re^{-t^2/2}(-it)^n \ud t &= (-1)^n \phi(x)H_n(x) , \text{ for all } x \in \R.
\end{align*}
Hence,
\begin{align*}
\int_{-\infty}^\infty   \re^{-itx} \re^{-t^2/2} (-i)^{j-k} t^{j+k} \frac{\ud t}{t}
    &= 2\pi i (-1)^{k+1} \int_0^\infty     \re^{-itx} \re^{-t^2/2} (-i t)^{j+k-1}   \ud t \\
    &= 2\pi i (-1)^{j} \phi (x) H_{j+k-1} (x) .
\end{align*}
So we obtain
\begin{align*}
\frac{i}{2 \pi}  \sum_{j=1}^m\frac{1}{j!} \int_{-\infty}^\infty   \re^{-itx} \re^{-t^2/2} g_{z,\alpha} (t)^j  \frac{\ud t}{t} 
&   =    \sum_{j=1}^m \frac{(-1)^{j+1}}{j!}  
\sum_{k=0}^j\binom{j}{k}z^{j-k}\left(\frac{\alpha}{2}\right)^k
\phi (x) H_{j+k-1} (x) .
\end{align*}
Thus we conclude that
\begin{align*}
& \sup_{x \in \R} \left| \Pr (Y\leq x) - \Phi(x) - \phi(x)\sum_{j=1}^m\frac{(-1)^{j+1}}{j!}\sum_{k=0}^j\binom{j}{k}z^{j-k}\left(\frac{\alpha}{2}\right)^kH_{j+k-1}(x) \right| \\
& {} \qquad\hspace{8cm} {} \leq 2^{2m}  \left( |\alpha|^{m+1} + \frac{\Gamma ( \frac{m+1}{2} )}{(m+1)!}  |z|^{m+1} \right) .
\end{align*}
The stated result now follows by re-expressing the sums over $k$ and $\ell = j+k-1$.
\end{proof}

 \begin{proof}[Proof of Proposition~\ref{prop:clt-with-error}]
Let $n \in \N$ and $t \in (0,\frac{1}{4(n+1)})$. Recall $R_{n,t}$ and $V_{n,t}$ from~\eqref{eq:R-def} and~\eqref{eq:V-def}. We apply the conditional Berry--Esseen result, Lemma~\ref{lem:b-e-conditional}, which has random centering and scaling, to obtain, for deterministic centering and scaling, 
\begin{align*}
\Pr \left( \frac{N_t-\Exp N_t}{\sqrt{ v(t)}} \leq x  \biggmid \cF_n \right)
  &= \Pr  \left( \frac{N_t-\Exp N_t - R_{n,t}}{\sqrt{V_{n,t}}} \leq 
\sqrt{\frac{v(t)}{V_{n,t}}}\left(x  - \frac{ R_{n,t}}{\sqrt{ v(t)}}\right)
\biggmid \cF_n \right) \\
&=\Phi \left(\sqrt{ \frac{v(t)}{V_{n,t}}}\left(x  - \frac{ R_{n,t}}{\sqrt{ v(t)}}\right) \right) +  \Delta_{n,t} (x) , \text{ for all } x \in \R,
\end{align*}
for $\cF_n$-measurable random variables $\Delta_{n,t} (x)$ satisfying 
$\sup_{x \in \R} | \Delta_{n,t} (x) | \leq C M_n^2 t^{-3/2}$, a.s., by \eqref{eq:b-e-conditional}. 
Recall  
that $\Var N_t = v(t) = s_0 / t$ for $t \in (0,1/2)$, by~\eqref{eq:var-t}. 
Define
\begin{equation}
    \label{eq:Z-alpha-def}
    \alpha_{n,t} := \frac{S_{n,t}}{v(t)}, \text{ and } Z_{n,t} := - \frac{R_{n,t}}{\sqrt{v(t)}} ,
\end{equation}
where $S_{n,t}$ is defined at~\eqref{eq:S-def}. 
Thus we can re-write the preceding bound as 
\begin{align}
\label{eq:b-e-conditional-5}
\sup_{x \in \R} \left| \Pr \left( \frac{N_t-\Exp N_t}{\sqrt{ v(t)}} \leq x  \biggmid \cF_n \right)
- \Phi \left( \frac{x+Z_{n,t}}{\sqrt{1-\alpha_{n,t}}} \right) \right| \leq \frac{C M_n^2}{t^{3/2}}, \as
\end{align}
Note that, by~\eqref{eq:Z-alpha-def} with the first inequality in~\eqref{eq:S-K-bound} and our choice $t < \frac{1}{4(n+1)}$, we have $\alpha_{n,t} = S_{n,t}/v(t) \leq 1/2$. 
Moreover, for $\ell \in \ZP$ we may express $\cE^\ell_{n,t}$ as defined at~\eqref{eq:E-def} in terms of $\alpha_{n,t}$ and $Z_{n,t}$ defined at~\eqref{eq:Z-alpha-def} via 
\begin{equation}
    \label{eq:E-alpha-Z}
 \cE^\ell _{n,t}   = \sum_{k=0}^{\lfloor(\ell+1)/2\rfloor}\frac{(-1)^{\ell-k} }{2^k k! (\ell-2k+1)!}  Z_{n,t}^{\ell-2k+1} \alpha_{n,t}^k. \end{equation}
 Then we apply  Lemma \ref{lem:new-edgeworth} with $\alpha = \alpha_{n,t} \in (-\infty, 1/2]$ and $z = Z_{n,t} \in \R$ to give 
\begin{align*}
 \left|  \Phi \left( \frac{x+Z_{n,t}}{\sqrt{1-\alpha_{n,t}}} \right) 
- \Phi(x) - 
\phi(x)\sum_{\ell=0}^{2m-1} H_\ell (x) \cE^\ell_{n,t} \right| 
\leq 2^{2m} \left( |\alpha_{n,t}|^{m+1} + \frac{\Gamma ( \frac{m+1}{2} )}{(m+1)!} |Z_{n,t}|^{m+1} \right). 
\end{align*}
 Taking expectations in the above display
 and using the bound~\eqref{eq:cramer}, we obtain
\begin{multline*}
\sup_{x \in \R} \left| \Exp  \Phi \left( \frac{x+Z_{n,t}}{\sqrt{1-\alpha_{n,t}}} \right)   
- \Phi(x) \right| \leq  \sum_{\ell=0}^{2m+1} 2^{\frac{\ell}{2}} \Gamma \bigl( \tfrac{\ell+1}{2} \bigr)  \left| \Exp \cE^\ell_{n,t}   \right| \\
 + 2^{2m} \left( \Exp ( |\alpha_{n,t}|^{m+1} ) + \frac{\Gamma ( \frac{m+1}{2} )}{(m+1)!} \Exp ( |Z_{n,t} |^{m+1} ) \right).
\end{multline*}
Hence we take expectations in~\eqref{eq:b-e-conditional-5}, noting that $\Exp ( M_n^2 ) = O (n^{-2})$ by Lemma~\ref{lem:M-moments}, to get
\begin{align}
\label{eq:main-series}
  \sup_{x \in \R} \left| \Pr \left( \frac{N_t -\Exp N_t}{\sqrt { \Var N_t}} \leq x  \right) - \Phi (x) \right| & \leq 
  \frac{C}{n^2 t^{3/2}} +   \sum_{\ell=0}^{2m+1} 2^{\frac{\ell}{2}} \Gamma \bigl( \tfrac{\ell+1}{2} \bigr)  \left| \Exp \cE^\ell_{n,t}   \right| \nonumber\\
& {}  \quad {} + 2^{2m}  \left( \Exp ( |\alpha_{n,t}|^{m+1} ) + \frac{\Gamma ( \frac{m+1}{2} )}{(m+1)!} \Exp ( |Z_{n,t} |^{m+1} ) \right).
  \end{align}
  Now suppose $t \in ( n^{-\nu}, \frac{1}{40(n+1)} )$, and take $\delta >0$ as specified in Proposition~\ref{prop:R-S-cross-moments}. 
    From Proposition~\ref{prop:R-S-cross-moments}\ref{prop:R-S-cross-moments-i} with $p=0$ and $q=m+1 \leq \delta \log n$, we have
    \[ \Exp ( | \alpha_{n,t} |^{m+1} ) \leq C t^{m+1} \Exp ( | S_{n,t} |^{m+1} ) 
    \leq C t^{m+1} \bigl[ (n^2 t)^{m+1} + C (40 n^2 t)^{m+\frac{1}{2}} \bigr] \leq C (40 n t)^{2m+2} ,
    \]
    since $n t < 1$. On the other hand, $\Exp ( |Z_{n,t}|^{m+1} ) \leq C t^{\frac{m+1}{2}} \Exp ( | R_{n,t} |^{m+1} )$. For $m -1 \in 2\ZP$, we have from Proposition~\ref{prop:R-S-cross-moments}\ref{prop:R-S-cross-moments-i} with $p=m+1 \leq \delta n$ and $q=0$,
    \[ \Exp ( | R_{n,t}|^{m+1} ) \leq \frac{ (m+1)!}{(\frac{m+1}{2} )!} (n^2 t)^{\frac{m+1}{2}} + ( \tfrac{m+3}{2} )! (40 n^2 t)^{\frac{m}{2}} . \]
    Using Stirling's formula, $\Gamma ( x) \sim \sqrt{ 2\pi x} ( x/\re )^x$, and the fact that $n t<1$,
    it follows that 
    \[ \frac{\Gamma ( \frac{m+1}{2} )}{(m+1)!} \Exp ( |Z_{n,t} |^{m+1} ) \leq C (40 n t)^{m} . \]
    For $m \in 2\ZP$, using Proposition~\ref{prop:R-S-cross-moments}\ref{prop:R-S-cross-moments-i} (again) with $p=m+2 \leq \delta n$ and $q=0$, plus Jensen's inequality, leads to the same conclusion. 
  Thus from~\eqref{eq:main-series} we get
  \begin{align}
\label{eq:main-series-2}
  \sup_{x \in \R} \left| \Pr \left( \frac{N_t -\Exp N_t}{\sqrt { \Var N_t}} \leq x  \right) - \Phi (x) \right| & \leq 
  \frac{C}{n^2 t^{3/2}} +   \sum_{\ell=0}^{2m+1} 2^{\frac{\ell}{2}} \Gamma \bigl( \tfrac{\ell+1}{2} \bigr)  \left| \Exp \cE^\ell_{n,t}   \right| 
  + C (80 n t)^m,
  \end{align}
  provided $m \leq \delta \log n$. Taking $m = \lfloor (\delta/3)  \log n \rfloor$ and $n t \leq c < 1/80$, we have $ (80 n t)^m \leq (80c)^{(\delta/4) \log n}$ for all $n$ large enough; thus provided $c \leq c_0 := \re^{-2/\delta}/80$ we have that $(80 nt)^m \leq C n^{-1/2}$. Then from~\eqref{eq:main-series-2} we conclude~\eqref{eq:clt-with-error}.
   \end{proof} 

With   Proposition~\ref{prop:clt-with-error} in hand, the remaining task in the proof of Theorem~\ref{thm:n-b-e} is to bound $\Exp  \cE^{\ell}_{n,t} $ from~\eqref{eq:E-def} and hence control  the sum on the right-hand side of~\eqref{eq:clt-with-error}. This is the purpose of the next result, which needs the full strength of Proposition~\ref{prop:R-S-cross-moments}. 

\begin{lemma}
    \label{lem:error-estimates}
    Suppose that $\nu \in (1, \frac{3}{2})$, and let $\delta >0$ be as in Proposition~\ref{prop:clt-with-error}.  
    \begin{thmenumi}[label=(\roman*)]
\item
\label{lem:error-estimates-i}
There exists  $C < \infty$ such that, for all $n \in \N$,
all $t \in ( n^{-\nu}, \frac{1}{40(n+1)} )$, and all $\ell -1  \in 2\ZP$ 
with $1 \leq \ell \leq (\delta /2) \log n$, 
\begin{equation}
    \label{eq:E-bound-odd}
\bigl| \Exp \cE^\ell _{n,t} \bigr| \leq 
    \frac{C}{(\frac{\ell+1}{2})!} \left( \frac{80n^2 t^2}{s_0} \right)^{\frac{\ell+1}{2}}
 (n^2 t)^{-1/2}.
\end{equation}  
\item
\label{lem:error-estimates-ii}
It holds that $\Exp \cE^0_{n,t} = 0$. Moreover, there exists  $C < \infty$ such that, for all $n \in \N$,
all $t \in ( n^{-\nu}, \frac{1}{40(n+1)} )$, and all $\ell  \in 2\ZP$ 
with $1 \leq \ell \leq ( \delta/2) \log n$, 
\begin{equation}
    \label{eq:E-bound-even}
\bigl| \Exp \cE^\ell _{n,t}  \bigr|  \leq 
    \frac{C}{( \frac{\ell}{2} )!} \left( \frac{80n^2 t^2}{s_0} \right)^{\frac{\ell}{2}}  n^{-1}. 
\end{equation}
\end{thmenumi}
    \end{lemma}
\begin{proof}
Suppose that $\nu \in (1, \frac{3}{2})$, and let $\delta >0$ be as in Proposition~\ref{prop:clt-with-error}. Take $n \in \N$
and $t \in ( n^{-\nu}, \frac{1}{40(n+1)} )$. 
Taking expectations in~\eqref{eq:E-def} and using~\eqref{eq:var-t}, we have
\begin{align*}
\Exp \cE^\ell_{n,t}   & = ( t/s_0)^{\frac{\ell+1}{2}} \sum_{k=0}^{\lfloor(\ell+1)/2\rfloor}\frac{(-1)^{k+1}}{2^k k! (\ell-2k+1)!}
\Exp \left( R_{n,t}^{\ell-2k+1} S_{n,t}^k \right). 
\end{align*}
The idea is to apply Proposition~\ref{prop:R-S-cross-moments}
with $p= \ell - 2k +1$ and $q=k$, so that $1 \leq p+q = \ell - k +1 \leq \delta \log n$ for all $n$ large enough.  
For part~\ref{lem:error-estimates-i}, suppose that $\ell -1 \in 2\ZP$. 
Then $\ell-2k+1 \in 2\ZP$, and so from
    Proposition~\ref{prop:R-S-cross-moments}\ref{prop:R-S-cross-moments-i}  we get 
\begin{align*}
 \left| \Exp \cE^\ell _{n,t}   - 
\left( \frac{n^2 t^2}{12s_0} \right)^{\frac{\ell+1}{2}} 
 \sum_{k=0}^{\frac{\ell+1}{2}} 
\frac{(-1)^{k+1} }{  k! \bigl( \frac{\ell+1}{2} - k \bigr)!} 
 \right|
  \leq 
    C   \left( \frac{40n^2 t^2}{s_0} \right)^{\frac{\ell+1}{2}}
     \sum_{k=0}^{\frac{\ell+1}{2}} \frac{ (n^2 t)^{-1/2}}{k! ( \frac{\ell+1}{2} -k )!},
     \end{align*}
where in the bound we used $ \sup_{p \in \N} (\frac{p}{2} +1 )!  ( \frac{p}{2} )! / p! < \infty$, as follows from~\eqref{eq:simple-binomial}.
Here
\[ 
\sum_{k=0}^{\frac{\ell+1}{2}} (-1)^{k+1} 
\frac{1}{  k! \bigl( \frac{\ell+1}{2} - k \bigr)!} 
 =
 \frac{1}{( \frac{\ell+1}{2} )!} \sum_{k=0}^{\frac{\ell+1}{2}} (-1)^{k+1} \binom{\frac{\ell+1}{2}}{k} = 0, 
\]
and, similarly,
\[ 
\sum_{k=0}^{\frac{\ell+1}{2}}  
\frac{1}{  k! \bigl( \frac{\ell+1}{2} - k \bigr)!} 
 =
 \frac{1}{( \frac{\ell+1}{2} )!} \sum_{k=0}^{\frac{\ell+1}{2}}   \binom{\frac{\ell+1}{2}}{k} =  \frac{2^{\frac{\ell+1}{2}}}{( \frac{\ell+1}{2} )!} .
\]
From here we get~\eqref{eq:E-bound-odd}.
For part~\ref{lem:error-estimates-ii}, suppose $\ell \in 2\ZP$ is even. 
Then $\ell-2k+1$ is odd, and so from
    Proposition~\ref{prop:R-S-cross-moments}\ref{prop:R-S-cross-moments-ii}   we get 
\begin{align*}
\left| \Exp \cE^\ell _{n,t}   - 
\left( \frac{n^2 t^2}{12s_0} \right)^{\frac{\ell}{2}} (t/s_0)^{1/2} 
 \sum_{k=0}^{\frac{\ell}{2}} 
\frac{2 k (-1)^{k+1} }{  k! \bigl( \frac{\ell}{2} - k \bigr)!} 
 \right| 
     \leq 
    C  \left( \frac{40n^2 t^2}{s_0} \right)^{\frac{\ell}{2}}
    ( t/s_0)^{1/2} \sum_{k=0}^{\frac{\ell}{2}} \frac{ (n^2 t)^{-1/2}}{k! ( \frac{\ell}{2} -k )!},
     \end{align*}
using $\sup_{p \in \N} (\frac{p+1}{2})! (\frac{p+1}{2})! /p! < \infty$, by~\eqref{eq:simple-binomial}. 
Now
\[ \sum_{k=0}^{\frac{\ell}{2}} 
\frac{  k (-1)^{k+1} }{  k! \bigl( \frac{\ell}{2} - k \bigr)!} 
= \frac{1}{(\frac{\ell}{2})!} \sum_{k=0}^{\frac{\ell}{2}} (-1)^{k+1} \binom{\frac{\ell}{2}}{k}  k
= 0 , \]
and from here we obtain~\eqref{eq:E-bound-even}.
\end{proof}

\begin{proof}[Proof of Theorem~\ref{thm:n-b-e}]
Let $\delta >0$ be as in Proposition~\ref{prop:R-S-cross-moments}. Proposition~\ref{prop:clt-with-error} shows that, for   constants $c_0 >0$ and $C < \infty$, for all $n \in \N$ and $t >0$ such that $n^{-5/4}   \leq nt \leq c_0$, say, the bound~\eqref{eq:clt-with-error} holds. 
In particular, 
taking $c_1 \in (0,c_0)$ for which $160c_1^2 / s_0 < 1/2$, 
the bound~\eqref{eq:clt-with-error} holds whenever $t >0$ and $n = n_t := \lfloor c_1/t \rfloor$ is sufficiently large.
For $\ell \leq 2 \lfloor (\delta/3) \log n \rfloor +1$, we have $\ell \leq (\delta/2) \log n$ for all $n$ large enough, and so Lemma~\ref{lem:error-estimates}, with $n = n_t$, shows that
\[ \bigl| \cE^\ell_{n_t,t} \bigr| \leq \frac{C}{\Gamma (\frac{\ell+1}{2} )} \left( \frac{80 c_1^2}{s_0} \right)^{\frac{\ell}{2}} n_t^{-1/2} , \text{ for all } 1 \leq \ell \leq (\delta/2) \log n_t.\]
Hence from~\eqref{eq:clt-with-error} we get
\[
  \sup_{x \in \R} \left| \Pr \left( \frac{N_t -\Exp N_t}{\sqrt { \Var N_t}} \leq x  \right) - \Phi (x) \right| \leq 
  \frac{C}{n_t^2 t^{3/2}} + 
C n_t^{-1/2} \sum_{\ell=0}^{\infty} 2^{\frac{\ell}{2}}  \left( \frac{80 c_1^2}{s_0} \right)^{\frac{\ell}{2}} \leq C \sqrt{t},
\]
by our choice of $n_t$ and of $c_1$.
\end{proof}

\begin{proof}[Proof of Theorem~\ref{thm:clt}]
Consider the interval $I_n :=  [ - \sqrt{n/\sigma^2} , \sqrt{n/\sigma^2}]$. 
We  claim that
it is a consequence of  Theorem~\ref{thm:n-b-e} and the relation $\Pr ( M_n \leq t) = \Pr (N_t \leq n)$ from~\eqref{eq:M-N-inverse}, which holds for all $n \in \N$
and all $t \in (0,1)$, that 
 there exists $C \in \RP$ such that
\begin{align}
\label{eq:clt-proof-2}
 \sup_{x \in I_n}   \left|  \Pr \left( \sqrt{ \frac{n^3}{\sigma^2}} \left( M_n - \frac{2}{n} \right)  \leq x  \right)  - \Phi ( x ) \right| 
     &   \leq  C  n^{-1/2}.
\end{align}
Assuming (for now) that~\eqref{eq:clt-proof-2} holds, we extend the bound over all $x \in \R$  using monotonicity and comparison to the small Gaussian tails. Indeed, we have from~\eqref{eq:clt-proof-2} that
\begin{align}
\label{eq:N-t-n-BE-3}
   \Pr \left( \sqrt{ \frac{n^3}{\sigma^2}} \left( M_n - \frac{2}{n} \right)  \leq - \sqrt{ n /\sigma^2}  \right)  
   \leq \Phi ( - \sqrt{ n /\sigma^2} ) +  O (n^{-1/2}) =  O (n^{-1/2}) ,
\end{align}
by standard tail bounds for~$\Phi$, and similarly for the positive tail. Then
\begin{align*}
    {} &
 \sup_{ x  < -  \sqrt{ n /\sigma^2}} \left|  \Pr \left( \sqrt{ \frac{n^3}{\sigma^2}} \left( M_n - \frac{2}{n} \right)  \leq x  \right)  - \Phi (x) \right|  \\
 & {} \qquad {} 
 \leq 
 \sup_{  x <-  \sqrt{ n /\sigma^2}}   \Pr \left( \sqrt{ \frac{n^3}{\sigma^2}} \left( M_n - \frac{2}{n} \right)  \leq x  \right) 
 +  \sup_{ x <  - \sqrt{ n /\sigma^2}}  \Phi (x) \\
 & {} \qquad  \leq \Pr \left( \sqrt{ \frac{n^3}{\sigma^2}} \left( M_n - \frac{2}{n} \right)  \leq - \sqrt{ n /\sigma^2}  \right)
 + \Phi ( - \sqrt{ n /\sigma^2} ) =  O (n^{-1/2}),
\end{align*}
by~\eqref{eq:N-t-n-BE-3}. A similar argument applies to the values of $x > \sqrt{ n /\sigma^2}$,
noting that
\begin{align*}
 \sup_{ x  >  \sqrt{ n /\sigma^2}} \left|  \Pr \left( \sqrt{ \frac{n^3}{\sigma^2}} \left( M_n - \frac{2}{n} \right)  \leq x  \right)  - \Phi (x) \right| 
 %\\
 %& {} \qquad {}
 =  \sup_{ x  >  \sqrt{ n /\sigma^2}} \left|  \Pr \left( \sqrt{ \frac{n^3}{\sigma^2}} \left( M_n - \frac{2}{n} \right)  > x  \right)  - \barPhi (x) \right| ,
\end{align*}
where $\barPhi (x) := 1- \Phi(x)$.
Thus Theorem~\ref{thm:clt} follows from~\eqref{eq:clt-proof-2}. It remains to verify the claim~\eqref{eq:clt-proof-2}, which we deduce from a careful inversion of Theorem~\ref{thm:n-b-e}.

For $n \in \N$ and $x \in I_n$, define 
\[ t_n (x) := \frac{2}{n} + x \sqrt{ \frac{\sigma^2}{n^3}} , \text{ and }
y_n (x) := \sqrt{\frac{2 t_n(x)}{\sigma^2}} \left( n - \frac{2}{t_n(x)} \right).
\]
Observe that $t_n (x) \geq 1/n$ for $x \in I_n$.
It follows from~\eqref{eq:M-N-inverse} that
\begin{equation}
    \label{eq:clt-inversion}
 \Pr \left( \sqrt{ \frac{n^3}{\sigma^2}} \left( M_n - \frac{2}{n} \right)  \leq x  \right) 
= \Pr \left( N_{t_n(x)} \leq n \right) .
\end{equation}
Also observe that
\begin{equation}
    \label{eq:y-n-x}
    y_n (x) = 
    x \sqrt{ \frac{2}{n t_n(x)}} = 
    x \left( 1 + \frac{x}{2} \sqrt{ \frac{\sigma^2}{n}} \right)^{-1/2} ,
\end{equation}
 the term in brackets in~\eqref{eq:y-n-x} is positive for $x \in I_n$,
so $y_n(x)$ has the same sign as~$x$, and indeed
\begin{equation}
    \label{eq:y-n-square-bounds}
    \frac{2 x^2}{3} \leq y_n(x)^2 \leq 2 x^2, \text{ for all } x \in I_n.
\end{equation}
We have, by Theorem~\ref{thm:n-b-e}, there exists $C \in \RP$ such that, for all $n \in \N$ and all $x \in \R$,
\begin{align*}
    \left| \Pr \left( N_{t_n(x)} \leq n \right) - \Phi ( y_n(x) ) \right| 
     &=  \left| \Pr \left( \sqrt{ \frac{2 t_n(x)}{\sigma^2} }\left(  N_{t_n(x)} - \frac{2}{t_n(x)} \right)\leq y_n(x) \right) - \Phi (y_n(x)) \right| \\
      &\leq  C  t_n (x)^{1/2}.
\end{align*}
Thus, by~\eqref{eq:clt-inversion} and the fact that $t_n (x) \geq 1/n$ for $x \in I_n$, we get
\begin{align}
\label{eq:clt-proof-1}
 \sup_{x \in I_n}   \left|  \Pr \left( \sqrt{ \frac{n^3}{\sigma^2}} \left( M_n - \frac{2}{n} \right)  \leq x  \right)  - \Phi ( y_n(x) ) \right| 
     &   \leq  C  n^{-1/2}.
\end{align}
It remains to compare $\Phi ( y_n (x))$ to $\Phi (x)$.
Write (as previously) $\phi$ for the standard Gaussian
  density. By the mean value theorem, there
exists $\theta = \theta_n (x) \in (0,1)$ such that
\begin{equation}
\label{eq:mvt}\Phi ( y_n (x)) - \Phi (x) 
= (y_n(x) -x ) \phi (x + \theta (y_n(x) - x)) .
\end{equation}
Here, for $x \in I_n$,
\[  \phi (x + \theta (y_n(x) - x)) 
\leq (2 \pi)^{-1/2} \exp \left\{ - \min ( x, y_n (x) )^2/2 \right\} \leq  (2 \pi)^{-1/2} \re^{- x^2 /3},
\]
by~\eqref{eq:y-n-square-bounds}. Moreover, from~\eqref{eq:y-n-x} we get
$| y_n (x) - x | \leq C x^2 / \sqrt{n}$ for all $x \in I_n $ and some constant $C \in \RP$.
Hence there is a constant $C \in \RP$ such that
\begin{equation}
\label{eq:Phi-x-y}
\bigl|  \Phi ( y_n (x)) - \Phi (x)  \bigr|
\leq C \frac{x^2 }{\sqrt{n}} \re^{-x^2/3} , \text{ for all } x \in I_n.
\end{equation}
Combining~\eqref{eq:clt-proof-1} and~\eqref{eq:Phi-x-y} we verify~\eqref{eq:clt-proof-2}.
\end{proof}

Finally, we give the proofs of the corollaries from Section~\ref{sec:branching}. 

\begin{proof}[Proof of Corollary~\ref{cor:cmj}]
This is direct from Theorem~\ref{thm:n-b-e} since, as explained in Section~\ref{sec:branching},
Kingman's embedding gives $T_t = N_{\re^{-t}}$, $t \in (0,\infty)$.
\end{proof}

\begin{proof}[Proof of Corollary~\ref{cor:brw}]  
Recalling the identification $\ell_n = \log (1/M_n)$, we have, for $x \in \R$,
\begin{align}
\label{eq:cor-change-1}
    \nonumber \Pr \left( \sqrt{\frac{4n}{\sigma^2}} \left( \ell_n - \log (n/2) \right) \geq x \right) 
     &= \Pr \left( M_n \leq \frac{2}{n} \exp \left( - \sqrt{\frac{\sigma^2}{4n}} x \right) \right) \\
     &= \Pr \left( \sqrt{\frac{ n^3}{\sigma^2}} \left( M_n - \frac{2}{n}  \right) \leq -b_n (x) \right), 
\end{align}
where
\[ b_n (x) := \sqrt{\frac{4n}{\sigma^2}} \left( 1 - \exp \left( - \sqrt{\frac{\sigma^2}{4n}} x \right)  \right) . \]
It follows that, for some $C < \infty$,
$| b_n (x) - x | \leq C x^2 / \sqrt{n}$,  for all $x \in \R$  and all $n\in \N$.
Hence there exists $\delta >0$ such that $| b_n (x) - x | \leq x/2$ for all $x \in [-\delta \sqrt{n}, \delta \sqrt{n}]$, and, in particular,
$b_n (\delta \sqrt{n} ) > (\delta/2) \sqrt{n}$ and
$b_n (-\delta \sqrt{n} ) < - (\delta/2) \sqrt{n}$.
Consequently, from~\eqref{eq:cor-change-1}, 
\begin{align*}
   \sup_{x \geq \delta \sqrt{n}}  \Pr \left( \sqrt{\frac{4n}{\sigma^2}} \left( \ell_n - \log (n/2) \right) \geq x \right) \leq \Pr \left( \sqrt{\frac{ n^3}{\sigma^2}} \left( M_n - \frac{2}{n}  \right) \leq -(\delta/2) \sqrt{n} \right) = O (n^{-1/2} ),
\end{align*}
by Theorem~\ref{thm:clt}; similarly for $x \leq - \delta\sqrt{n}$. On the other hand, by~\eqref{eq:cor-change-1} and Theorem~\ref{thm:clt}, there exists $C< \infty$ such that, for all $n \in \N$ and all
$x \in [ - \delta \sqrt{n}, \delta \sqrt{n} ]$,
\begin{align*}
    \left| \Pr \left( \sqrt{\frac{4n}{\sigma^2}} \left( \ell_n - \log (n/2) \right) \geq x \right)  - \Phi ( - x) \right| & \leq \left| \Phi ( b_n(x)) - \Phi (x) \right| + Cn^{-1/2}.
\end{align*}
By the mean value theorem, similarly to~\eqref{eq:mvt}, for all $n \in \N$, 
\[    \sup_{x \in [ - \delta \sqrt{n}, \delta \sqrt{n} ]} \left| \Phi ( b_n(x)) - \Phi (x) \right| \leq (2 \pi)^{-1/2} \sup_{x \in \R} \re^{-x^2/8} \bigl| b_n (x) - x \bigr|  = O (n^{-1/2} ), \]
which completes the proof of~\eqref{eq:brw-1}. The proof of~\eqref{eq:brw-2} is direct from  the fact that
\[ \Pr \left( r_n - \log (n^2/2) < x \right) = \Pr  \left( \frac{n^2 m_n}{2} > \re^{-x} \right) , \text{ for all } x \in \R, \]
and applying Theorem~\ref{thm:small-gap}.
\end{proof}

\appendix 

\section{Moments of partial sums}
\label{sec:appendix}

We give two auxiliary results needed in the proof of Lemma~\ref{lem:K-R-W-moments}. While we suspect that neither result is new, we have not been able to locate a reference. Lemma~\ref{lem:K-statistic-moments} gives first-order expansions of moments of sums of i.i.d.~summands, with quantitative error bounds, that has some similarities to 
the  Macinkiewicz--Zygmund and Rosenthal inequalities~\cite[pp.~146--153]{gut}.

\begin{lemma}
\label{lem:K-statistic-moments}
Let $\xi$ be a random variable with $\Exp ( |\xi|^k ) < \infty$ for every $k \in \ZP$.
Consider
$\Xi_{n} := \sum_{i=1}^{n} \xi_i$,
where $\xi_1,\xi_2, \ldots$ are i.i.d.\ copies of $\xi$.
Write $\mu_k (n) := \Exp ( \Xi_n^k )$, and 
\begin{equation}
    \label{eq:A-k-def}
    A_k := \bigl[ ( \Exp \xi )^k + \Exp ( | \xi|^{k-1} ) \bigr] k^2,
    \text{ and } A'_k := \bigl[ ( \Exp ( \xi^2 ) )^k + \Exp ( | \xi|^{2k-2} ) \bigr] k^2.
\end{equation}
\begin{thmenumi}[label=(\roman*)]
\item
\label{lem:K-statistic-moments-i}
Suppose that $\Exp \xi =0$. Then for all $k \in \ZP$ and all $n \in \N$,
\[ \left| \mu_{2k} (n) - ( \Exp ( \xi^2) )^k \frac{ (2k)!}{2^k k!} n^k \right| \leq A'_{k} \frac{ (2k)!}{2^k k!}  n^{k-1} .\]
\item
\label{lem:K-statistic-moments-ii}
Suppose that $\Exp \xi > 0$. 
 Then for all $k \in \ZP$ and all $n \in \N$,
 \[ \left| \mu_{k} (n) - ( \Exp  \xi )^k n^k \right| \leq A_k n^{k-1} .\]
\end{thmenumi}
\end{lemma}
\begin{proof}
First we prove~\ref{lem:K-statistic-moments-ii}.
Similarly to the proof of Lemma~\ref{lem:small-gaps-K} (but with an index shift), write $I_{n,k} = \{1,\ldots,n\}^k$
and $I_{n,k}^\circ$ for vectors in $I_{n,k}$ with no two coordinates the same. Then, by the fact that the $\xi_i$ are i.i.d., 
\[ \Exp \sum_{(i_1, \ldots, i_k) \in I^\circ_{n,k} } \xi_{i_1} \cdots \xi_{i_k} 
= ( \Exp \xi )^k | I^\circ_{n,k} | ,
\]
 On the other hand, 
 given $(i_1, \ldots, i_k) \in I_{n,k}$ for which
 there are $\ell \in \{ 1, \ldots,  k\}$ distinct indices appearing in multiplicities $m_1 + \cdots + m_\ell = k$, by independence,  
 \[ \Exp | \xi_{i_1} \cdots \xi_{i_{k}}  |
 = \prod_{j=1}^\ell \Exp ( | \xi |^{m_j} )
 \leq \prod_{j=1}^\ell \bigl( \Exp ( | \xi |^{k} ) \bigr)^{m_j/k} = \Exp ( | \xi |^{k} ) , \]
 where we used Lyapunov's inequality for the middle step. Hence 
\[ \Exp \sum_{(i_1, \ldots, i_k) \in I_{n,k} \setminus I^\circ_{n,k} } \bigl| \xi_{i_1} \cdots \xi_{i_k} \bigr|
\leq \binom{k}{2} \Exp \sum_{(i_1, \ldots, i_{k-1}) \in I_{n,k-1} } \bigl| \xi_{i_1} \cdots \xi_{i_{k-1}} \bigr| \leq k^2 \Exp ( | \xi |^{k-1} ) n^{k-1}.
\]
Moreover, by~\eqref{eq:I-n-k-bound}, we have $0 \leq n^k - | I_{n,k}^\circ | \leq k^2 n^{k-1}$. It follows that
\[ \left| \mu_k (n) - ( \Exp \xi )^k n^k \right| \leq \left[ ( \Exp \xi )^k + \Exp ( |\xi|^{k-1} ) \right] k^2 n^{k-1} = A_k n^{k-1}, \]
where $A_k$ is as defined at~\eqref{eq:A-k-def}. This yields part~\ref{lem:K-statistic-moments-ii}. 
     For part~\ref{lem:K-statistic-moments-i}, note that
    \[ \mu_{2k} (n) = \Exp \sum_{1 \leq i_1, \ldots, i_{2k} \leq n} \xi_{i_1} \cdots \xi_{i_{2k}} 
    = \frac{(2k)!}{2^k k!} \Exp \sum_{1 \leq i_1, \ldots, i_{k} \leq n} \xi^2_{i_1} \cdots \xi^2_{i_{k}} ,
        \]
        since the product has expectation zero whenever $(i_1, \ldots, i_{2k} )$
        has a coordinate which appears exactly once,
        and the combinatorial factor comes from the number
        of ways of pairing up coordinates. The rightmost expression in the
        last display is $\Exp ( \widetilde{\Xi}^k_{n} )$
        for $\widetilde{\Xi}_{n} = \sum_{i=1}^{n} \xi_i^2$. 
       Thus part~\ref{lem:K-statistic-moments-i} follows from
       part~\ref{lem:K-statistic-moments-ii},
       with $A_k'$ in~\eqref{eq:A-k-def} being the quantity $A_k$ but with $\xi^2$ in place of $\xi$.
\end{proof}

\begin{lemma}
    \label{lem:binomial-product}
    Let $X \sim \Bin{n}{1/2}$. Then, for every $\alpha \in \ZP, \beta \in \ZP$ with $\max (\alpha ,\beta ) \geq 1$,
    \[  
    \left| \Exp \bigl( X^\alpha (n-X)^\beta \bigr) -  (n/2)^{\alpha+\beta} \right| \leq \max (\alpha, \beta) n^{\alpha+\beta-(1/2)}, \text{ for all } n \in \N.
    \]
\end{lemma}\begin{proof}
Write $Q_n (x) := x^\alpha (n-x)^\beta$ and $Z_n :=X -(n/2)$.
If $\alpha \geq 1$, it follows from the mean value theorem that $| (1+y)^\alpha - 1 | \leq \alpha  2^{\alpha-1} |y|$, $|y| \leq 1$.
If $\alpha \in [0,1]$, then $g(y):= ((1+y)^\alpha - 1)/y$ has $g'(y) \in (-\infty,0)$
for all $|y| < 1$, and hence $\sup_{y \in [-1,1]} | g(y) | = g(-1) = 1$. Combining the two cases gives, for every $\alpha \geq 0$,
\begin{equation}
    \label{eq:y-inequality}
    \left| (1+y)^\alpha - 1 \right| \leq c_\alpha |y|, \text{ for all } | y| \leq 1, \text{ where } c_\alpha := \max ( 1, \alpha 2^{\alpha-1} ) .
\end{equation}
Using~\eqref{eq:y-inequality} we observe that, for $0 \leq x \leq n$, 
 \begin{align*}
 \left| x^\alpha - (n/2)^\alpha \right| & = (n/2)^\alpha \left| \left(1 + \frac{x - (n/2)}{n/2} \right)^\alpha - 1 \right| 
 \leq c_\alpha (n/2)^{\alpha-1} | x- (n/2) | .
 \end{align*}
 Hence, for $0 \leq x \leq n$,
 \begin{align}
 \label{eq:Q-n-diff}
     | Q_n (x) - Q_n (n/2) | & \leq (n/2)^\alpha  \left| (n-x)^\beta - (n/2)^\beta \right| + (n/2)^\beta  \left| x^\alpha - (n/2)^\alpha \right|  \nonumber\\
     & \leq ( c_\alpha + c_\beta )   (n/2)^{\alpha+\beta-1} | x- (n/2) | .
 \end{align}
    Then, using~\eqref{eq:Q-n-diff}, 
\begin{align*} \left| \Exp  \bigl( X^\alpha (n-X)^\beta \bigr) - Q_n (n/2) \right|
& \leq \Exp \left| Q_n (X) - Q_n (n/2) \right| \leq 
     ( c_\alpha + c_\beta )  (n/2)^{\alpha+\beta-1} \Exp | Z_n |. \end{align*}
By Lyapunov's inequality, $\Exp |Z_n| \leq ( \Exp ( Z_n^2 ) )^{1/2} = ( \Var X )^{1/2} = \sqrt{ n/4}$, so that
\begin{align*} \left| \Exp  \bigl( X^\alpha (n-X)^\beta \bigr) - Q_n (n/2) \right|
& \leq  
    2^{-\alpha-\beta} ( c_\alpha + c_\beta )  n^{\alpha+\beta-(1/2)} . \end{align*}
Considering separately the cases $\min (\alpha, \beta) \geq 1$ and $\min (\alpha,\beta) =0$ (and using $2^{-x} \leq x/2$ for $x \geq 1$) it is not
hard to verify that $2^{-\alpha-\beta} (c_\alpha + c_\beta ) \leq \max(\alpha, \beta)$ as long as $\max (\alpha, \beta) \geq 1$.
        \end{proof}

\section*{Acknowledgements}

AW was supported by EPSRC grant EP/W00657X/1.
Part of this work was undertaken  during the programme ``Stochastic systems for anomalous diffusion'' (July--December 2024) hosted by the  Isaac Newton Institute, under EPSRC grant EP/Z000580/1. The authors are grateful to an anonymous referee for pointing out an error in an earlier draft.

\end{document}